\documentclass[10pt, utf8]{article}

\usepackage{microtype}
\usepackage{graphicx}
\usepackage{subfigure}
\usepackage{booktabs}

\usepackage{amsmath,amsbsy,amsgen,amscd,amssymb,amsthm,amsfonts,stmaryrd}
\usepackage{mathtools}

\usepackage{bm}

\usepackage{microtype}      

\usepackage[colorlinks,citecolor=blue]{hyperref}
\usepackage{url}            

\usepackage[usenames,dvipsnames]{xcolor}
\usepackage{graphicx}

\graphicspath{{art/}}

\definecolor{dark-gray}{gray}{0.3}
\definecolor{dkgray}{rgb}{.4,.4,.4}
\definecolor{dkblue}{rgb}{0,0,.5}
\definecolor{medblue}{rgb}{0,0,.75}
\definecolor{rust}{rgb}{0.5,0.1,0.1}

\hypersetup{urlcolor=rust}
\hypersetup{citecolor=blue}
\hypersetup{linkcolor=blue}

\newtheorem{theorem}{Theorem}
\newtheorem{lemma}[theorem]{Lemma}

\theoremstyle{definition}
\newtheorem{remark}[theorem]{Remark}
\newtheorem{example}{Example}

\numberwithin{equation}{section}
\numberwithin{theorem}{section}

\renewcommand{\phi}{\varphi}

\DeclareFontFamily{OT1}{pzc}{}
\DeclareFontShape{OT1}{pzc}{m}{it}{<-> s * [1.200] pzcmi7t}{}
\DeclareMathAlphabet{\mathpzc}{OT1}{pzc}{m}{it}

\DeclareMathOperator{\dist}{dist}
\DeclareMathOperator{\diag}{diag}
\DeclareMathOperator{\blkdiag}{blkdiag}
\DeclareMathOperator{\dom}{dom}

\DeclareMathOperator{\prox}{prox}

\DeclareMathOperator{\gra}{gra}

\newtheorem{assumption}{Assumption \!\!}
\newcommand{\bs}{\boldsymbol}
\usepackage{xcolor,colortbl}
\newcommand{\Dsigma}{\sigma}

\makeatletter
\newtheorem*{rep@theorem}{\rep@title}
\newcommand{\newreptheorem}[2]{%
	\newenvironment{rep#1}[1]{%
		\def\rep@title{#2 \ref{##1}}%
		\begin{rep@theorem}}%
		{\end{rep@theorem}}}
\makeatother

\newreptheorem{theorem}{Theorem}
\newreptheorem{lemma}{Lemma}

\usepackage{lscape}
\usepackage{makecell}

\usepackage[sort, numbers]{natbib}
\usepackage{algorithm,algorithmic}
\usepackage[margin=1in]{geometry}

\usepackage[nameinlink]{cleveref}
\crefname{lemma}{lemma}{lemmas}
\crefname{assumption}{assumption}{assumptions}

\title{On the convergence of stochastic primal-dual hybrid gradient}
\author{Ahmet Alacaoglu$^\mathsection$
$\quad$ Olivier Fercoq$^\dagger$
$\quad$ Volkan Cevher$^\ddagger$ 
\\ \\
$^\mathsection$University of Wisconsin-Madison, USA \\
$^\dagger$LTCI, T\'el\'ecom Paris,  Institut Polytechnique de Paris, France\\
$^\ddagger$LIONS, Ecole Polytechnique F\'ed\'erale de Lausanne, Switzerland 
}

\begin{document}

\maketitle

\begin{abstract}
In this paper, we analyze the recently proposed stochastic primal-dual hybrid gradient (SPDHG) algorithm and provide new theoretical results.
In particular, we prove almost sure convergence of the iterates to a solution with convexity and linear convergence with further structure, using standard step sizes independent of strong convexity or other regularity constants.
In the general convex case, we also prove the $\mathcal{O}(1/k)$ convergence rate for the ergodic sequence, on expected primal-dual gap function.
Our assumption for linear convergence is metric subregularity, which is satisfied for strongly convex-strongly concave problems in addition to many nonsmooth and/or nonstrongly convex problems, such as linear programs, Lasso, and support vector machines.
We also provide numerical evidence showing that SPDHG with standard step sizes shows a competitive practical performance against its specialized strongly convex variant SPDHG-$\mu$ and other state-of-the-art algorithms including variance reduction methods. 
\end{abstract}

\section{Introduction}

Stochastic primal-dual hybrid gradient (SPDHG) algorithm is proposed by Chambolle et al.~\cite{chambolle2018stochastic}, for solving the optimization problem
\begin{equation}\label{eq: prob_temp}
\min_{x\in\mathcal{X}} \sum_{i=1}^n f_i(A_i x) + g(x),
\end{equation}
where $f_i\colon \mathcal{Y}_i \to \mathbb{R} \cup \{+\infty\}$ and $g\colon \mathcal{X} \to \mathbb{R} \cup \{+\infty\}$ are proper, lower semicontinuous (l.s.c.), convex functions and $f$ is defined as the separable function such that $f(y)=\sum_{i=1}^n f_i(y_i)$.
$A_i\colon \mathcal{X} \to \mathcal{Y}_i$ is a linear mapping and $A$ is defined such that $(Ax)_i = A_i x$.

The classical approaches provide numerical solutions to \eqref{eq: prob_temp} via primal-dual methods. 
In particular, a common strategy is to have coordinate-based updates for the separable dual variable~\cite{zhang2017stochastic, chambolle2018stochastic}.
These methods show competitive practical performance and are proven to converge linearly under the assumption that $f_i^\ast, \forall i$ and $g$ are $\mu_i$ and $\mu_g$-strongly convex functions, respectively.
Step sizes of these methods in turn depend on $\mu_i, \mu_g$ to obtain linear convergence. SPDHG belongs to this class. 

Chambolle et al.\ provide convergence analysis for SPDHG under various assumptions on the problem template~\cite{chambolle2018stochastic}. Indeed, 
SPDHG is a variant of celebrated primal-dual hybrid gradient (PDHG) method~\cite{chambolle2011first, chambolle2016ergodic} where the main difference is stochastic block updates for dual variables at each iteration.
In the general convex case,~\cite{chambolle2018stochastic} proved that a particular Bregman distance between the iterates of SPDHG and any primal-dual solution converges almost surely to $0$ and the ergodic sequence has a $\mathcal{O}(1/k)$ rate for this quantity. Note however that this result does not imply the almost sure convergence of the sequence to a solution, in general.
However, this result does not give guarantees on the expected primal-dual gap function (see~\eqref{eq: pd_gap},~\eqref{eq: h_def}), which is the standard optimality measure.
If $f_i^\ast$ and $g$ are strongly convex functions, SPDHG-$\mu$, which is a variant of SPDHG with step sizes depending on strong convexity constants, is proven to converge linearly~\cite[Theorem 6.1]{chambolle2018stochastic}.
Estimation of strong convexity constants can be challenging in practice, restricting the use of SPDHG-$\mu$.

Since its introduction, SPDHG has been popular in practice, especially in computational imaging, with implementations in different software packages~\cite{ehrhardt2017faster,papoutsellis2021core,liu2019real,kazantsev2019versatile}.
Despite the practical interest, fundamental theoretical results regarding the convergence of SPDHG remained open, including almost sure convergence, $\mathcal{O}(1/k)$ convergence rate for expected primal-dual gap and adaptive linear convergence.

In its most basic form, step sizes of SPDHG are determined using $\| A_i \|$ and probabilities of selecting coordinates~\cite{chambolle2018stochastic}.
It is often observed in practice that the last iterate of PDHG or SPDHG with these step sizes has competitive practical performance. Yet, only ergodic rates are known for this method with restrictive assumptions~\cite{chambolle2016ergodic, chambolle2018stochastic}.
In this paper, we analyze SPDHG with standard step sizes and provide new theoretical results, paving the way for explaining its fast convergence behavior in practice.

\subsection{Our contributions} We prove the following results for SPDHG:
\paragraph{General convex case} We prove that the iterates of SPDHG converge almost surely to a solution.
For this purpose, we introduce a representation of SPDHG as a fixed point operator in a duplicated space. For the ergodic sequence, 
we show that SPDHG has $\mathcal{O}(1/k)$ rate of convergence for the expected primal-dual gap.
To prove this result, we introduce a generic technique that is applicable to other stochastic primal-dual coordinate descent algorithms.
Moreover, we prove the same rate for objective residual and feasibility for linearly constrained problems.
\paragraph{Metrically subregular case} When the problem is metrically subregular (see~Section~\ref{sec: ms}), we prove that SPDHG has linear convergence with standard step sizes, depending only on $A_i$ and probabilities for selecting coordinates.
Our result shows that without any modification, basic SPDHG adapts to problem structure and attains linear rate when the assumption holds.
\paragraph{Practical performance} We show that SPDHG shows a robust and competitive practical performance compared to SPDHG-$\mu$ of~\cite{chambolle2018stochastic} and other state-of-the-art methods including variance reduction and primal-dual coordinate descent methods.

We summarize our results and compare with those of~\cite{chambolle2018stochastic} in~\Cref{tab:2} (Page \pageref{tab:2}).

\section{Preliminaries}\label{sec: prelim}
\subsection{Notation}\label{sec: notation}
We assume that $\mathcal{X}$ and $\mathcal{Y}$ are Euclidean spaces and that $\mathcal{Y} = \prod_{i=1}^n \mathcal{Y}_i$.
We define $\mathcal{Z} = \mathcal{X} \times \mathcal{Y}$ and $z=(x, y) \in \mathcal{Z}$.
For positive definite $Q$, we use $\langle x, y \rangle_Q = \langle Qx, y\rangle$ for denoting weighted inner product and $\| x\|_Q^2 = \langle Qx, x \rangle$ for weighted Euclidean norm.
We overload these notations to also write for a vector $\sigma$ with $\sigma_i > 0$, $\|y\|_\sigma^2 = \langle y, \diag(\sigma) y \rangle$.
For a set $\mathcal{C}$, and positive definite $Q$, distance of a point $x$ to $\mathcal{C}$, measured in $\|\cdot\|_Q$ is defined as $\dist^2_Q(x, \mathcal{C}) = \min_{y\in \mathcal{C}} \| x-y\|^2_Q = \|x-\mathcal{P}_{\mathcal{C}}^Q(x)\|_Q^2$, where we have defined the projection operator $\mathcal{P}$ implicitly.
When $Q=I$, we drop the subscript and write $\dist(x, \mathcal{C})$.
For $\sigma \in \mathbb{R}^n$, we use the elementwise inverse $\sigma^{-1} = (\sigma_1^{-1}, \dots, \sigma_n^{-1})$.
Domain of a function $h$ is denoted as $\dom{h}$. 
We encode constraints using the indicator function: $\delta_{\{b\}}(x) = 0$ if $x=b$ and $\delta_{\{b\}}(x) = +\infty$ if $x \neq b$.

Given a vector $x$, we access $i^{\text{th}}$ element as $x_i$. We define $e(i)\in\mathcal{Y}$ such that $e(i)_{j} = 1$, if $j=i$ and $e(i)_j = 0$, if $j\neq i$.
Moreover, we use $E(i) = e(i) e(i)^\top$.
Unless used with a subscript, $1$ in Kronecker products denotes $1_n\in\mathbb{R}^n$, all-ones vector.

Given a vector $x\in\mathcal{X}$, we use bold symbol $\bs x$ to denote the duplicated version of this vector, which consists of $n$ \emph{copies} of $x$, and the corresponding space is denoted by $\bs{\mathcal{X}}=\mathcal{X}^n$. Similarly, the duplicated dual space is $\bs{\mathcal{Y}}=\mathcal{Y}^n$ and $\bs{\mathcal{Z}}=\bs{\mathcal{X}}\times\bs{\mathcal{Y}}$. The copies might be the same, or different, depending on how $\bs x$ is set. To access $i^{\text{th}}$ copy, we use the notation $\bs x(i) \in \mathcal{X}$.
For the operator $T\colon \bs{\mathcal{Z}} \to \bs{\mathcal{Z}}$, and a duplicated vector $\bs q\in\bs{\mathcal{Z}}$, we denote the output as $T(\bs q) = \binom{T_x(\bs q)}{T_y(\bs q)}$. For example, $i^{\text{th}}$ primal copy is denoted as $T_x(\bs q)(i) \in\mathcal{X}$. Similarly, for the $i^{\text{th}}$ primal copy in $\bs q$, we use $\bs q_x(i)\in\mathcal{X}$. To access $i^{\text{th}}$ primal and dual copies, we use $\bs q(i) \in\mathcal{Z}$.

For example, when we pick one coordinate at a time, we can set $\mathcal{X}=\mathbb{R}^d$, $\mathcal{Y}=\mathbb{R}^n$, which would result in the duplicated spaces $\bs{\mathcal{X}}=\mathbb{R}^{dn}$, $\bs{\mathcal{Y}}=\mathbb{R}^{n^2}$, and $\bs{\mathcal{Z}} = \mathbb{R}^{dn+n^2}$.

Probability of selecting an index $i\in\{1, \dots, n\}$ is denoted as $p_i > 0$, with $\sum_{i=1}^n p_i = 1$.
We  define $P=\diag(p_1, \dots, p_n)$ and $\underline{p} = \min_i p_i$.
Notation $\mathcal{F}_k$ defines the filtration generated by randomly selected indices $\{ i_1, \dots, i_{k-1} \}$.
Let $\mathbb{E}_k\left[\cdot\right] := \mathbb{E}\left[ \cdot \mid \mathcal{F}_k \right]$ denote the conditional expectation with respect to $\mathcal{F}_k$.

The proximal operator of a function $h$ is defined as
\begin{equation}\label{eq: prox_def}
\text{prox}_{\tau, h} (x) = \arg\min_{u\in\mathcal{X}} h(u) + \frac{1}{2} \| u-x\|^2_{\tau^{-1}}.
\end{equation}
The Fenchel conjugate of $h$ is defined as
$h^\ast(y) = \sup_{z\in \mathcal{X} } \langle z, y\rangle - h(z).$
\subsection{Solution}
Using Fenchel conjugate,~\eqref{eq: prob_temp} is cast as the saddle point problem
\begin{equation}\label{eq: prob_minmax}
\min_{x\in\mathcal{X}} ~ \sup_{y\in\mathcal{Y}}~ \sum_{i=1}^n \langle A_i x, y_i \rangle - f^\ast_i(y_i) + g(x).
\end{equation}
A primal-dual solution $(x^\star, y^\star) \in \mathcal{Z^\star}$ is characterized as
\begin{equation}\label{eq: kkt}
0\in\begin{bmatrix}A^\top y^\star + \partial g(x^\star) \\ Ax^\star - \partial f^\ast(y^\star) \end{bmatrix} = F(x^\star, y^\star).
\end{equation}
Given the functions $g$ and $f^\ast$ as in~\eqref{eq: prob_minmax}, we define
\begin{align}
&D_g({x}; \bar{z}) = g({x}) - g(\bar{x}) + \langle A^\top \bar{y}, {x}-\bar{x} \rangle, \label{eq: bregman_g}\\
&D_{f^\ast}({y}; \bar{z}) = f^\ast({y}) - f^\ast(\bar{y}) - \langle A \bar{x}, {y}-\bar{y} \rangle. \label{eq: bregman_f}
\end{align}

When $\bar{z}=z^\star=(x^\star, y^\star)$, with $z^\star$ denoting a primal-dual solution as defined in~\eqref{eq: kkt}, we have that~\eqref{eq: bregman_g} and~\eqref{eq: bregman_f} are Bregman distances generated by functions $g(x)$ and $f^\ast(y)$.
In this case, these Bregman distances measure the distance between ${x}$ and $x^\star$, and ${y}$ and $y^\star$, respectively.
Given $z$, $D_h(z; z^\star)$ is the Bregman distance generated by $h(z)=g(x)+f^\ast(y)$, to measure the distance between $z$ and $z^\star$.
Moreover, the primal-dual gap function can be written as $G(z) = \sup_{\bar{z}\in\mathcal{Z}}D_{f^\ast}(x; \bar{z}) + D_g(y; \bar{z})$.

\subsection{Metric subregularity}\label{sec: ms}
For Euclidean spaces $\mathcal{U}, \mathcal{V}$ and a set valued mapping $F\colon \mathcal{U} \rightrightarrows \mathcal{V}$, we denote the graph of $F$ by $\gra F = \{ (u, v) \in \mathcal{U} \times \mathcal{V} \colon v \in Fu \}$.
We say that $F$ is metrically subregular at $\bar{u}$ for $\bar{v}$, with $(\bar{u}, \bar{v}) \in \gra F$, if there exists $\eta_0 > 0$ with a neighborhood of subregularity $\mathcal{N}(\bar{u})$ such that:
\begin{equation}\label{eq: def_ms}
\dist(u, F^{-1} \bar{v}) \leq \eta_0 \dist(\bar{v}, Fu),~~ \forall u \in \mathcal{N}(\bar{u}).
\end{equation}
If $\mathcal{N}(\bar{u})=\mathcal{U}$, then $F$ is globally metrically subregular~\cite{dontchev2009implicit}.
Absence of metric subregularity is signaled by $\eta_0 = +\infty$.
This assumption is used in the context of deterministic and stochastic primal-dual algorithms in~\cite{liang2016convergence, drusvyatskiy2018error, latafat2019new}.

In the paper we shall study how the metric subregularity of the Karush-Kuhn-Tucker (KKT) operator $F$ in~\eqref{eq: kkt} implies linear convergence of SPDHG.

Metric subregularity of $F$ holds in following cases: 
\begin{enumerate}
	\item $f_i^\ast$ and $g$ are strongly convex functions, since $\mathcal{N}(\bar z) = \mathcal{Z}$. 
	\item The problem~\eqref{eq: prob_temp} is  defined with piecewise linear quadratic (PLQ) functions and $\dom g$ and $\dom f^\ast$ are compact sets, in which case $\mathcal{N}(\bar z)=\dom g \times \dom f^\ast$. In particular the domain of a PLQ function can be represented as the union of finitely many polyhedral sets and in each set, the function is a quadratic~(see~\cite[Definition IV.3]{latafat2019new}). Problems with PLQ functions include Lasso, support vector machines, linear programs, etc.
\end{enumerate}
\begin{remark}\label{rem: cmp_plq}
In the first example above, compact domains are not needed since metric subregularity holds globally for these problems.
One can also relax strong convexity in the first case to weaker conditions (see~\cite{latafat2018primal}).
Importantly, compact domain assumption is only needed in the second example mentioned above in this paper, for PLQs.
The reason, as we see in Theorem~\ref{eq: thm_lin} is the lack of control on the low probability event that the trajectory makes an excursion far away.
The same assumption for proving linear convergence of another primal-dual coordinate descent method is also needed in~\cite{latafat2019new}.
\end{remark}

\section{Algorithm}
The algorithm SPDHG is given as~\Cref{alg: main_spdhg}.
\begin{algorithm}
\begin{algorithmic}
\STATE \textbf{Input:} Pick step sizes $\sigma_i, \tau$ by~\eqref{eq: ss_rules} and~$x^{0} \in \mathcal{X}$, $y^0=y^1=\bar{y}^1 \in \mathcal{Y}$. Given $P=\diag(p_1, \dots, p_n)$.
\FOR{$k=1, 2, \dots$}
\STATE $x^{k} = \text{prox}_{\tau, g}(x^{k-1} -\tau A^\top \bar{y}^k)$
\STATE Draw $i_k\in \{1,\dots, n\}$ such that $\Pr(i_k = i) = p_i$.
\STATE $y^{k+1}_{i_k} = \text{prox}_{\sigma_{i_k}, f_{i_k}^\ast}(y^k_{i_k} + \sigma_{i_k} A_{i_k} x^{k})$
\STATE $y^{k+1}_{i} = y^k_i,~~~ \forall i \neq i_k$ 
\STATE $\bar{y}^{k+1} = y^{k+1} + P^{-1} (y^{k+1} - y^k)$,
\ENDFOR
\end{algorithmic}
\caption{Stochastic PDHG (SPDHG) \cite[Algorithm 1]{chambolle2018stochastic}}
\label{alg: main_spdhg}
\end{algorithm}

\begin{remark}
We use serial sampling of blocks in our analysis for the ease of notation.
We can extend our results with other samplings by using expected separable overapproximation (ESO) inequality as in~\cite{chambolle2018stochastic}.
\end{remark}

We use the standard step size rules for primal and dual step sizes~\cite{chambolle2018stochastic}:
\begin{equation}\label{eq: ss_rules}
p_i^{-1}\tau\sigma_i \| A_i \|^2 \leq \gamma^2 < 1.
\end{equation}

\begin{assumption}\label{as: asmp1} We have the following assumptions concerning~\eqref{eq: prob_temp}.
\begin{enumerate}
\item $f_i$ and $g$ are proper, lower semicontinuous (l.s.c.), convex functions.
\item The set of solutions to~\eqref{eq: prob_temp} is nonempty.
\item Slater's condition holds, namely $0\in \mathrm{ri} (\dom{f} - A\dom{g})$ where $\mathrm{ri}$ stands for relative interior~\cite{bauschke2011convex}.
\end{enumerate}
\end{assumption}
Slater's condition is a standard sufficient assumption for strong duality, used frequently for primal-dual methods~\cite{bauschke2011convex, chambolle2011first, chambolle2018stochastic, latafat2019new, tran2018smooth, fercoq2019coordinate}. Strong duality ensures that a dual solution exists in~\eqref{eq: prob_minmax} and the set of primal-dual solutions is characterized by~\eqref{eq: kkt}.

\section{Convergence}
We start with a lemma analyzing one iteration behavior of the algorithm.
This lemma is essentially the same as~\cite[ Lemma 4.4]{chambolle2018stochastic} up to minor modifications and is included for completeness, with its proof in Section~\ref{sec: pf_first_lem}.
We first introduce some notations.
\begin{equation}\label{eq: v_vk_defin_main}
\begin{aligned}
V(z) &= \frac{1}{2} \| x\|^2_{\tau^{-1}} + \frac{1}{2} \| y\|^2_{\sigma^{-1}P^{-1}} + \langle Ax, P^{-1}y \rangle, \\
V_k(x, y) &= \frac{1}{2} \| x \|^2_{\tau^{-1}} -  \langle Ax, P^{-1}(y^k - y^{k-1}) \rangle + \frac{1}{2} \| y^k - y^{k-1}\|^2_{\sigma^{-1}P^{-1}}
+ \frac{1}{2} \| y \|^2_{\sigma^{-1}P^{-1}}.
\end{aligned}
\end{equation}
We also define the full dimensional dual update
\begin{equation}
\hat{y}^{k+1}_i = \prox_{\sigma_i, f^\ast_i}(y_i^k + \sigma_i A_i x^k),~~~\forall i \in \{1, \dots, n \}.\label{eq: haty_defin}
\end{equation}
\begin{lemma}\label{lem: one_iteration}
Let~\Cref{as: asmp1} hold.
It holds for SPDHG that, $\forall x\in\mathcal{X}, \forall y\in\mathcal{Y}$,
\begin{equation}\label{eq: one_it_lem}
D_g(x^k; z) + D_{f^\ast}(\hat{y}^{k+1}; z) \leq V_k(x^{k-1}-x, y^k - y) 
-\mathbb{E}_k \left[ V_{k+1}(x^k-x, y^{k+1}-y) \right] - V(z^k - z^{k-1}).
\end{equation}
Moreover, with $C_1 = 1-\gamma$, under the step size rules in~\eqref{eq: ss_rules}, we have
\begin{align}
V(z^k - z^{k-1}) &\geq C_1 \left( \frac{1}{2} \| x^k - x^{k-1}\|^2_{\tau^{-1}} + \frac{1}{2} \| y^k - y^{k-1} \|^2_{\sigma^{-1}P^{-1}} \right), \label{eq: v_lw_bd} \\
V_k(x, y) &\geq C_1 \left( \frac{1}{2} \| x \|^2_{\tau^{-1}} + \frac{1}{2} \| y^k - y^{k-1}\|^2_{\Dsigma^{-1}P^{-1}} \right) + \frac{1}{2} \| y\|^2_{\Dsigma^{-1}P^{-1}}.\label{eq: vk_lw_bd}
\end{align}
\end{lemma}

\subsection{Almost sure convergence}\label{sec: as}
In this section, we present the almost sure convergence of the iterates of SPDHG to a solution of~\eqref{eq: prob_temp}.

We start by introducing an equivalent representation of SPDHG that is instrumental in our proofs.
The motivation of this representation can be seen as similar to~\cite{he2012convergence}, where the focus was on PDHG.
In particular, this representation shifts the primal update so that the algorithm can be written as a fixed point operator.
Since $\bar{y}^{k+1}$ depends on the selected index $i_k$ at iteration $k$, the operator $T$ is defined such that all the possible values of $\bar{y}^{k+1}$ and consequently, of $x^{k+1}$ are captured.
\begin{lemma}\label{lem: defT}
Let us define $T\colon \bs{\mathcal{Z}} \to \bs{\mathcal{Z}}$ that to $(\bs x,\bs y)$ associates $(\bs {\hat x}, \bs{\hat y})$ such that $\forall i\in\{1, \dots, n\}$,
\begin{align*}
& \bs{\hat y}(i) = {\prox}_{\sigma, f^\ast}(\bs{y}(i) + \diag(\sigma) A \bs{x}(i)) \\
& \bs{\bar{y}}(i) = \bs{y}(i) + (1 + p_i^{-1}) (\bs{\hat y} (i)_i - \bs{y}(i)_i)e(i) \\
& \bs{\hat x}(i) = {\prox}_{\tau, g}(\bs{x}(i) - \tau A^\top \bs{\bar{y}}(i))
\end{align*}
where $\bs{x}(i)\in\mathcal{X}$, $\bs{y}(i) \in \mathcal{Y}$.

The fixed points of $T$ are of the form $(\bs x(i), \bs y(i))$ such that $(\bs x(i), \bs y(i))\in\mathcal{Z}^\star$, $\forall i\in\{1, \dots, n\}$. Moreover, 
\begin{equation}
\left(x^{k+1}, \hat{y}^{k+1}\right) = \left(T_x(1 \otimes x^k, 1\otimes y^k)(i_k), T_y(1\otimes x^k, 1\otimes y^k)(1)\right).\notag
\end{equation}
We also denote 
\begin{align}
\bar{S} &= \blkdiag(\tau^{-1}I_{dn\times dn}, I_{n\times n}\otimes \sigma^{-1}), \notag \\
\bar{P} &= \blkdiag(p_1 I_{d\times d}, \dots, p_n I_{d\times d}, p_1 I_{n\times n}, \dots, p_n I_{n\times n}).\notag
\end{align}
We then have,
\[
\|T(1\otimes x^k, 1\otimes y^k) - (1\otimes x^k, 1\otimes y^k)\|_{\bar{S}\bar{P}}^2 = \mathbb E_k\left[\|x^{k+1} - x^k\|^2_{\tau^{-1}} + \|y^{k+1} - y^k \|_{\Dsigma^{-1}P^{-1}}^2\right]\!.
\]
\end{lemma}
Before presenting the proof of the lemma, we use an example to illustrate the notation and the main idea behind it.
\begin{example}
\normalfont 
Let $d=1$, $n=2$, then $\bs x = \binom{\bs x(1)}{\bs x(2)} \in\mathbb{R}^2$, $\bs y = \binom{\bs y(1)}{\bs y(2)}\in\mathbb{R}^4$, and
\begin{align*}
\bar S &= \diag(\tau^{-1}, \tau^{-1}, \sigma_1^{-1}, \sigma_2^{-1}, \sigma^{-1}_1, \sigma_2^{-1}) \in \mathbb{R}^{6\times 6}, \\
\bar P &= \diag(p_1, p_2, p_1, p_1, p_2, p_2) \in \mathbb{R}^{6\times 6}.
\end{align*}
Then, we have by letting $\bs x = 1\otimes x^k$, $\bs y = 1\otimes y^k$,
\begin{alignat*}{2}
\bs {\hat y}(1) &= \prox_{\sigma, f^\ast}(y^k + \diag(\sigma) A x^k ),~~~~~~~~&&\bs {\hat y}(2) = \prox_{\sigma, f^\ast}(y^k + \diag(\sigma) A x^k ), \\
\bs{\bar y}(1) &= y^k + (1+p_1^{-1})\begin{bmatrix} \bs{\hat y}(1)_1 - y^k _1 \\ 0\end{bmatrix}, &&\bs{\bar y}(2) = y^k + (1+p_2^{-1})\begin{bmatrix} 0 \\ \bs{\hat y}(2)_2 - y^k _2 \end{bmatrix}, \\
\bs{\hat x}(1) &= \prox_{\tau, g}(x^k - \tau A^\top \bs{\bar y}(1)), &&\bs{\hat x}(2) = \prox_{\tau, g}(x^k - \tau A^\top \bs{\bar y}(2)).
\end{alignat*}
We have $T(1\otimes x^k, 1\otimes y^k) = \left(\begin{bmatrix} \bs{\hat x}(1) \\ \bs{\hat x}(2) \end{bmatrix}, \begin{bmatrix} 
\bs{\hat y}(1) \\ \bs{\hat y}(2) \end{bmatrix}\right)$. By using the definition of $\hat y^{k+1}$ in~\Cref{lem: one_iteration}, we see that $(x^{k+1}, \hat y^{k+1}) = (\bs{\hat x}(1), \bs{\hat y}(1))$ if $i_k = 1$ and $(x^{k+1}, \hat y^{k+1}) = (\bs{\hat x}(2), \bs{\hat y}(1))$ if $i_k = 2$. Note that we can take any copy of $\bs{\hat y}$ as $\bs{\hat y}(1) = \bs{\hat y}(2)$.
Moreover, depending on $i_k$, one obtains $y^{k+1}$ from $\hat y^{k+1}$ with a coordinate-wise update, as given in SPDHG (see~\Cref{alg: main_spdhg}).
\end{example}

\begin{proof}[Proof of~\Cref{lem: defT}]
Let $(\bs x, \bs y)$ be a fixed point of $T$. Then it follows that \\$\bs y(i) = \prox_{\sigma, f^*}(\bs y(i) + \diag(\sigma) A\bs x(i))$, $\forall i$, $\bs {\bar y}(i) =\bs y(i)$, $\forall i$ and $\bs x(i) = \prox_{\tau, g}(\bs x(i) -\tau A^\top \bs y(i))$, $\forall i$. Hence, optimality conditions for each $i$ are the same as~\eqref{eq: kkt}. Therefore fixed points of $T$ are such that $(\bs x(i), \bs y(i))\in\mathcal{Z}^\star, \forall i$.

The equality $(x^{k+1}, \hat{y}^{k+1}) = (T_x(1\otimes x^k, 1\otimes y^k)(i_k), T_y(1\otimes x^k, 1\otimes y^k)(1))$ is just another way to write the algorithm.
Since when inputted $(1\otimes x^k, 1\otimes y^k)$, $T$ outputs $(1\otimes \hat{y}^{k+1})$ for the dual variable, we can simply take first copy for $\hat{y}^{k+1}$. 

For the last result, we use $\| \hat{y}^{k+1} - y^k \|^2_{\Dsigma^{-1}} = \mathbb{E}_k \big[ \| y^{k+1}-y^k \|^2_{\Dsigma^{-1}P^{-1}} \big]$ to show
\begin{align*}
\|T&(1\otimes x^k, 1\otimes y^k) - (1\otimes x^k, 1\otimes y^k)\|^2_{\bar{S}\bar{P}} \notag\\
&= \sum_{i=1}^n \left(\| T_x(1\otimes x^k,1 \otimes y^k)(i) - x^k \|_{\tau^{-1}}^2 p_i + \| T_y(1\otimes x^k,1 \otimes y^k)(i) - y^k \|_{\Dsigma^{-1}}^2p_i\right) \\
&= \sum_{i=1}^n \left(\| T_x(1\otimes x^k,1 \otimes y^k)(i) - x^k \|_{\tau^{-1}}^2 p_i \right) + \| \hat{y}^{k+1} - y^k \|_{\Dsigma^{-1}}^2\big(\sum_{i=1}^n p_i\big) \\
&= \mathbb E_k\left[\|x^{k+1} - x^k\|_{\tau^{-1}}^2 + \|y^{k+1} - y^k\|_{\Dsigma^{-1}P^{-1}}^2\right],
\end{align*}
where we also used that $\sum_{i=1}^n p_i = 1$.
\end{proof}

We proceed with the main theorem of this section. We present the main ideas and ingredients that make the proof possible in the following proof sketch.
The details of the proof using classical arguments from~\cite{combettes2015stochastic,bertsekas2011incremental,iutzeler2013asynchronous} are deferred to~Section~\ref{sec: full_proof_as_thm}.
Let us define 
\begin{equation}\label{eq: gi4}
{\Delta}^k = V_{k+1}(x^k - x^\star, y^{k+1}-y^\star).
\end{equation}
\begin{theorem}\label{thm: conv_as}
Let~\Cref{as: asmp1} hold. Then, it holds that $\mathbb{E}[V_k(x^{k-1}-x^\star, y^k - y^\star)] \leq  \Delta^0$, $\sum_{k=1}^{\infty} \mathbb{E}[V(z^k - z^{k-1})] \leq \Delta^0$.
Moreover, almost surely, there exists $(x^\star, y^\star) \in \mathcal{Z}^\star$, such that the iterates of SPDHG satisfy $(x^k,y^k) \to (x^\star, y^\star)$.
\end{theorem}
\begin{proof}[Proof sketch]
On~\eqref{eq: one_it_lem}, we pick $(x,y)=({x}^\star,y^\star)$ and by convexity, $D_g(x^k; z^\star) \geq 0$, $D_{f^\ast}(\hat{y}^{k+1}; z^\star) \geq 0$. Next, by using the definition of $\Delta^k$, we write~\eqref{eq: one_it_lem} as
\begin{equation*}
\mathbb{E}_k \left[ {\Delta}^k \right] \leq \Delta^{k-1} - V(z^k - z^{k-1}).
\end{equation*}
We apply Robbins-Siegmund lemma~\cite[Theorem 1]{robbins1971convergence} to get that almost surely, $\Delta^k$ converges to a finite valued random variable and $V(z^k-z^{k-1})\to 0$. 
Consequently, by~\eqref{eq: v_lw_bd}, $\| y^k - y^{k-1}\|$ converges to $0$ almost surely. 
Since almost surely, $\Delta^k$ converges and $\| y^k - y^{k-1}\|$ converges to $0$, we have that $\| z^k - z^\star \|$ converges almost surely.

Next, we denote $\bs q^k = (1\otimes x^k, 1\otimes y^k)$ and use the arguments in~\cite[Proposition 2.3]{combettes2015stochastic},~\cite[Theorem 1]{fercoq2019coordinate} to argue that there exists a probability $1$ set $\Omega$ such that for every $z^\star \in \mathcal{Z}^\star$ and for every $\omega \in \Omega$, $\|z^k(\omega) - z^\star\|$ converges and $\| T(\bs q^k(\omega)) - \bs q^k(\omega) \| \to 0$.
As for every $\omega\in\Omega$, $(z^k(\omega))_k$ is bounded, we denote by $\tilde{z}=(\tilde{x}, \tilde{y})$ one of its cluster points.
Then, we denote $\tilde{\bs q} = (1\otimes \tilde{x}, 1\otimes \tilde{y})$ and have that $\tilde{\bs q}$ is a cluster point of $(\bs q^k(\omega))_k$.

The key step in our proof that enables the result is the fixed point characterization of $T$ in~\Cref{lem: defT}.
With this result, we derive $\tilde{z}\in\mathcal{Z}^\star$ as $\tilde{\bs q}$ is a fixed point of $T$.

To sum up, we have shown that at least on some subsequence $z^k(\omega)$ converges to $\tilde{z}\in\mathcal{Z}^\star$.
As for every $\omega\in\Omega$ and $z^\star\in\mathcal{Z}^\star$, $\|z^k(\omega)-z^\star\|$ converges, the result follows.
\end{proof}

\subsection{Linear convergence}\label{sec: lin_conv}
The standard approach for showing linear convergence with metric subregularity is to obtain a \textit{Fejer-type inequality} of the form~\cite{latafat2019new}
\begin{equation}
\mathbb{E}_k \left[ d(z^{k+1} - z^\star) \right] \leq  d(z^{k} - z^\star)- V(T(z^{k})-z^{k}),
\end{equation}
for suitably defined distance functions $d$, $V$ and operator $T$.
However, as evident from~\eqref{eq: one_it_lem} and the definition of $V_{k+1}$, one iteration result of SPDHG does not fit into this form.
When $x=x^\star, y = y^\star$, $V_{k+1}(x^k - x^\star, y^{k+1}-y^\star)$ does not only measure distance to solution, but also the distance of subsequent iterates $y^{k+1}$ and $y^k$.
In addition, $V_{k+1}$ includes $x^k-x^\star$ and $y^{k+1}-y^\star$ rather than $x^{k+1}-x^\star$ and $y^{k+1}-y^\star$, which further presents a challenge due to asymmetry, for using metric subregularity.
Therefore, an intricate analysis is needed to control the additional terms and handle the asymmetry in $V_{k+1}$.
In addition,~\Cref{lem: defT} is a necessary tool to identify $T$.

We need the following notation and lemma which builds on Lemma~\ref{lem: defT}, for easier computations with metric subregularity.
For the operators, we adopt the convention in~\cite{latafat2019new}.
Operator $C$ is the concatenation of subdifferentials, $M$ is the skew symmetric matrix that is formed using matrix $A$.
Operator $F$ is the KKT operator and $\bs H$ is the ``metric'' that helps us write the algorithm in proximal point form (see~\Cref{lem: defT}).
Due to duplication in~\Cref{lem: defT}, we need duplicated versions of $C$ and $M$. Consistent with the notation of~\Cref{lem: defT} (also see~Section~\ref{sec: notation}), we use boldface to denote operators in the duplicated space.
\begin{lemma}\label{lem: ops}
Under the notations of Lemma~\ref{lem: defT}, to write compactly the operation of $T$, let us define the operators
\begin{align*}
&C\colon (x, y) \mapsto (\partial g(x), \partial f^\ast(y)), \\
&M\colon (x, y) \mapsto (A^\top y, -Ax), \\
&\bs C\colon (\bs x, \bs y) \mapsto (\partial g(\bs x(1)), \dots, \partial g(\bs x(n)), \partial f^\ast(\bs y(1)), \dots, \partial f^\ast(\bs y(n))), \\
&\bs M\colon (\bs x, \bs y) \mapsto ( A^\top \bs y(1), \dots, A^\top \bs y(n), -A\bs x(1), \dots, -A\bs x(n)), \\
& F = C+M,
\end{align*}
and
\begin{align*}
\bs H\colon (\bs x, \bs y) \mapsto \big(&\tau^{-1} \bs x(1) + A^\top (1+p_1^{-1}) E(1)  \bs y(1), \ldots, \\ &\tau^{-1}\bs x(n) + A^\top(1+p_n^{-1}) E(n)  \bs y(n), \Dsigma^{-1} \bs y(1), \ldots, \Dsigma^{-1} \bs y(n) \big)\;.
\end{align*}
Let $\bs q^k\!=\!(1\otimes x^k, 1\otimes y^k)$, $\hat{\bs q}^{k+1} = T(\bs q^k)$ and
$\hat{z}^{k+1}=(x^{k+1}, \hat{y}^{k+1}) = (\hat{\bs q}^{k+1}_x(i_k), \hat{\bs q}^{k+1}_y(1))$.
Then, we have $(\bs H-\bs M)\bs q^k \in (\bs C+ \bs H)\hat{\bs q}^{k+1}$, $(\bs M-\bs H)(\hat{\bs q}^{k+1}-\bs q^k) \in (\bs C+\bs M)\hat{\bs q}^{k+1}$, 
\begin{equation*}
\mathbb{E}_k \left[ \dist^2(0,F\hat{z}^{k+1})\right] = \mathbb{E}_k \left[ \dist^2(0,(C+M)\hat{z}^{k+1}) \right] = \dist^2_{\bar{P}}(0, (\bs C+\bs M)\hat{\bs q}^{k+1}).
\end{equation*}
\end{lemma}
\begin{proof}
We start by the representation in~\Cref{lem: defT} by incorporating the update of $\bar{y}^{k+1}$, and recalling the definition of $E(i) = e(i)e(i)^\top$, $\forall i \in \{ 1, \ldots, n \}$
\begin{align*}
 \hat y(i) &= {\prox}_{\sigma, f^\ast}(y(i) + \diag(\sigma) A x(i)) \\
 \hat x(i) &= {\prox}_{\tau, g}(x(i) - \tau A^\top \left[ y(i) + (1+p_i^{-1})E(i)(\hat{y}(i) - y(i)) \right]) \\
 & = {\prox}_{\tau, g}(x(i) - \tau A^\top(1+p_i^{-1})E(i)\hat{y}(i) + \tau A^\top (-I_{n\times n} + (1 + p_i^{-1})E(i))y(i)).
\end{align*}
We now use the definition of proximal operator to obtain
\begin{align*}
& \sigma^{-1}y(i) + A x(i) \in \partial f^\ast(\hat{y}(i)) + \sigma^{-1} \hat{y}(i) \\
& \tau^{-1} x(i) - A^\top \! y(i) + A^\top \!(1\!+\!p_i^{-1})E(i)y(i) \in \partial g(\hat{x}(i)) + \tau^{-1}\hat{x}(i) + A^\top\!(1\!+\!p_i^{-1})E(i)\hat{y}(i).
\end{align*}
We identify
\begin{align*}
\bs H\bs q\!=\!\begin{bmatrix} \tau^{-1}\bs x(1) + A^\top(1+p_1^{-1})E(1)\bs y(1) \\ \vdots \\ \tau^{-1}\bs x(n) + A^\top(1+p_n^{-1})E(n)\bs y(n) \\ \Dsigma^{-1}\bs y(1) \\ \vdots \\ \Dsigma^{-1} \bs y(n) \end{bmatrix},
\bs M \bs q = \begin{bmatrix} A^\top \bs y(1) \\ \vdots \\ A^\top \bs y(n) \\ - A\bs x(1) \\ \vdots \\ - A\bs x(n) \end{bmatrix}, \bs C {\bs q} = \begin{bmatrix} \partial g(\bs x(1)) \\ \vdots \\ \partial g(\bs x(n)) \\ \partial f^\ast(\bs y(1)) \\ \vdots \\ \partial f^\ast(\bs y(n)) \end{bmatrix}.
\end{align*}
We set ${\bs q} = {\bs q}^k$ and $\hat{\bs q} = \hat{\bs q}^{k+1}$ and use the definition of $T$ in~\Cref{lem: defT} to obtain the first inclusion.

The second inclusion follows by adding to both sides $\bs M\hat{\bs q}^{k+1}$ and rearranging.

For the equality, we write
\begin{align*}
\mathbb{E}_k \left[ \dist^2(0, (C+M)\hat{z}^{k+1}) \right] &= \sum_{i=1}^n \dist^2(0, (C+M) {\hat{\bs{q}}^{k+1}(i)})p_i \notag \\
&= \dist^2_{\bar{P}}(0, (\bs C+\bs M)\hat{\bs{q}}^{k+1}),
\end{align*}
where the first equality follows by $\hat{z}^{k+1}=(x^{k+1}, \hat{y}^{k+1}) = (\hat{\bs q}^{k+1}_x(i_k), \hat{\bs q}^{k+1}_y(1))$ and the second equality is by the definitions of $C$, $M$, $\bs C$, and $\bs M$ and $\hat{\bs q}^{k+1}_y(i) = \hat{\bs q}^{k+1}_y(1)$, $\forall i$.
\end{proof}

We continue with our assumption for linear convergence (see~Section~\ref{sec: ms}).
\begin{assumption}\label{as: ms}
Metric subregularity holds for $F$ (see~\eqref{eq: kkt},~Section~\ref{sec: ms}) at all $z^\star\in\mathcal{Z^\star}$ for $0$ with constant $\eta > 0$ using $\|\cdot \|_S$ with
$S = \diag(\tau^{-1}1_p, \sigma_1^{-1}, \dots, \sigma_n ^{-1})$, and the neighborhood of regularity $\mathcal{N}(z^\star)$ contains $\hat{z}^k, \forall k$.
\end{assumption}

We present our main theoretical development in the next theorem, which states that SPDHG with step sizes in~\eqref{eq: ss_rules} attains linear convergence with~\Cref{as: ms}.
The proof idea is to utilize the term $-V(z^k-z^{k-1})$ in~\eqref{eq: one_it_lem} to obtain contraction. For this, we have to use the results of~\Cref{lem: defT,lem: ops} to write this term with the fixed point characterization given in~\Cref{lem: defT}, which allows using metric subregularity.

We denote 
\begin{equation*}
(x_\star^{k-1}, y_\star^{k}) = \arg\min_{(x, y)\in \mathcal{Z}^\star} V_k(x^{k-1} - x, y^k - y),
\end{equation*}
which exists since $V_k$ is a nonnegative quadratic function.
We define (cf.~\eqref{eq: gi4})
\begin{align*}
\Delta^k &= V_{k+1}(x^k - x^{k}_\star, y^{k+1} - y^{k+1}_\star), \\
 \Phi^k &= \Delta^k - \frac{C_1}{4\zeta} \| y^k - y_\star^k \|^2_{\Dsigma^{-1}} \geq 0.
\end{align*}
\begin{theorem}\label{eq: thm_lin}
Let Assumptions~\ref{as: asmp1} and~\ref{as: ms} hold.
Then it holds that 
\begin{equation}\label{eq: fejer_type}
\mathbb{E}_k \left[ {\Delta}^k \right] \leq \Delta^{k-1} - V(z^k - z^{k-1}),
\end{equation}
and
\begin{align*}
\mathbb{E} \bigg[ \frac{C_1}{2} \| x^k - x^k_\star \|^2_{\tau^{-1}} + \frac{1}{2} \| y^{k+1} - y^{k+1}_\star \|&^2_{\Dsigma^{-1}P^{-1}}\bigg] \leq (1-\rho)^k 2\Phi^0,
\end{align*}
where, $\rho = \frac{C_1\underline{p}}{2\zeta}$, $\zeta = 2+2\eta^2\| \bs H-\bs M \|^2$, $C_1 = 1-\gamma$.
\end{theorem}
\begin{proof}
Starting from the result of~\Cref{lem: one_iteration}, we have
\begin{align}\label{eq: res_lem1}
D_g(x^k; z) + D_{f^\ast}(\hat{y}^{k+1}; z) &\leq -\mathbb{E}_k \big[ V_{k+1}(x^k-x, y^{k+1}-y) \big] \notag \\
&+V_k(x^{k-1}-x, y^k - y) - V(z^k - z^{k-1}).
\end{align}
We pick $x={x}^{k-1}_\star$, $y=y^k_\star$, with $z^k_\star = (x^{k-1}_\star, y^{k}_\star)$ and use convexity to get $D_g(x^k; {z}^k_\star) \geq 0$ and $D_{f^\ast}(\hat{y}^{k+1}; z^{k}_\star) \geq 0$.
In addition, we define
\begin{align*}
\Delta^{k-1} = V_k(x^{k-1}-x^{k-1}_\star, y^k - y^k_\star) \\
\tilde{\Delta}^k = V_{k+1}(x^k - x^{k-1}_\star, y^{k+1}-y^k_\star).
\end{align*}
We use these definitions in~\eqref{eq: res_lem1} to write
\begin{equation*}
\mathbb{E}_k \left[ \tilde{\Delta}^k \right] \leq \Delta^{k-1} - V(z^k - z^{k-1}).
\end{equation*}
By definition of $(x^{k}_\star, y^{k+1}_\star)$, we have $\Delta^k \leq \tilde{\Delta}^k$, which implies that
\begin{equation*}
\mathbb{E}_k \left[ {\Delta}^k \right] \leq \Delta^{k-1} - V(z^k - z^{k-1}).
\end{equation*}
Recursion of this inequality gives boundedness of the iterates $x_k$ and $y_k$, in expectation. 
However, it is not possible to derive sure boundedness of the sequence. 
Without sure boundedness, the set that includes $x_k, y_k$ depends on the specific trajectory of the algorithm, and it is not possible to find a set independent of these.
As metric subregularity holds for PLQs with a bounded neighborhood (see~Section~\ref{sec: ms}), we cannot utilize this result and this is the main reason for the need for bounded domains in this case. This assumption would ensure sure boundedness of the sequence, which gives us a suitable set to use for using metric subregularity assumption for PLQs.

We recall $S = \diag(\tau^{-1}1_p, \sigma_1^{-1}, \dots, \sigma_n ^{-1})$, $\bar{S}$ and $\bar{P}$ are as defined in~\Cref{lem: defT}, and $\dist^2_{S}(z^k, \mathcal{Z^\star}) = \| z^k - \mathcal{P}^S_{\mathcal{Z}^\star}(z^k) \|^2_{S}= \| x^k-\tilde{x}^k_\star\|^2_{\tau^{-1}} + \| y^k - y^k_\star \|^2_{\Dsigma^{-1}}$ where $\tilde{x}^k_\star$ is the projection of $x^k$ onto the set of solutions with respect to norm $\| \cdot \|_{\tau^{-1}}$.
We now use~\Cref{as: ms} stating that $F=C+M$ is metrically subregular at $\mathcal{P}_{\mathcal{Z}^\star}^S(\hat{z}^{k+1})$ for $0$. 
We recall, $\bs q^k = (1\otimes x^k, 1\otimes y^k)$ and $\hat{\bs q}^{k+1} = T(\bs q^k)$ and estimate as
\begin{align}
 \| x^k - \tilde{x}^k_\star \|^2_{\tau^{-1}} + \| y^k - y^k_\star &\|^2_{\sigma^{-1}} = \dist^2_S(z^k, \mathcal{Z}^\star) \leq  \mathbb{E}_k \left[ \| z^k - \mathcal{P}^S_{\mathcal{Z}^\star}(\hat{z}^{k+1}) \|^2_S \right] \notag \\
 &\leq 2 \mathbb{E}_k \left[ \| z^k - \hat{z}^{k+1} \|_S^2 \right] + 2\mathbb{E}_k \left[ \| \hat{z}^{k+1} - \mathcal{P}^S_{\mathcal{Z}^\star}(\hat{z}^{k+1}) \|_S^2 \right],
\label{eq: metr_sub0}
\end{align}
where the first inequality is due to the definition of $\dist_{S}^2(z^k, \mathcal{Z}^\star)$.
Next, we estimate the second term on RHS
\begin{multline}
2\mathbb{E}_k \left[ \| \hat{z}^{k+1} - \mathcal{P}^S_{\mathcal{Z}^\star}(\hat{z}^{k+1}) \|_S^2 \right] \leq 2\eta^2 \mathbb{E}_k\left[ \dist^2_S(0, (C+M)\hat{z}^{k+1})\right] \\
= 2\eta^2 \dist^2_{\bar{S}\bar{P}}(0, (\bs C+ \bs M)\hat{\bs q}^{k+1}) \leq 2\eta^2 \|\bs M-\bs H\|^2 \|\hat{\bs q}^{k+1} - \bs q^k \|_{\bar{S}\bar{P}}^2, 
\end{multline}
with the first inequality being due to metric subregularity~of $C+M$~(see Remark~\ref{rem: metric_sub}) since $\dist_S^2(\hat z^{k+1}, \mathcal{Z}^\star) = \| \hat z^{k+1} - \mathcal{P}_{Z^\star}^S(\hat z^{k+1})\|^2_S$. First equality and second inequality are by~\Cref{lem: ops} and Cauchy-Schwarz inequality.
Joining the estimates give
\begin{equation}
 \| x^k - \tilde{x}^k_\star \|^2_{\tau^{-1}} + \| y^k - y^k_\star \|^2_{\Dsigma^{-1}} \leq 2 \mathbb{E}_k \left[ \| z^k - \hat{z}^{k+1}\|_S^2 \right] 
 + 2\eta^2 \|\bs M-\bs H\|^2 \|\hat{\bs q}^{k+1} - \bs q^k \|_{\bar{S}\bar{P}}^2.
\end{equation}
First, we use $\| \hat{y}^{k+1} - y^k \|^2_{\Dsigma^{-1}} = \mathbb{E}_k \left[ \| y^{k+1}-y^k \|^2_{\Dsigma^{-1}P^{-1}} \right]$ to estimate
\begin{align}
\mathbb{E}_k \left[ \| z^k - \hat{z}^{k+1} \|^2_{S}\right] &=   \mathbb{E}_k \left[ \| x^{k+1} - x^{k} \|^2_{\tau^{-1}} \right] +  \| \hat{y}^{k+1} - y^k \|^2_{\Dsigma^{-1}} \notag \\
&= \mathbb{E}_k \left[ \| x^{k+1} - x^{k} \|^2_{\tau^{-1}} +  \| {y}^{k+1} - y^k \|^2_{\Dsigma^{-1}P^{-1}} \right].\label{eq: metr_sub1}
\end{align}
Second, we use Lemma~\ref{lem: defT} to obtain
\begin{align}
\| \hat{\bs q}^{k+1} - \bs q^k \|^2_{\bar{S}\bar{P}} &=  \| T(1\otimes x^k, 1\otimes y^k) - (1\otimes x^k, 1\otimes y^k) \|^2_{\bar{S}\bar{P}} \notag \\
&= \mathbb{E}_k \left[  \| x^{k+1} - x^k \|^2_{\tau^{-1}} + \| y^{k+1} - y^k \|^2_{\Dsigma^{-1}P^{-1}} \right].\label{eq: metr_sub2}
\end{align}
We combine~\eqref{eq: metr_sub1} and~\eqref{eq: metr_sub2} in~\eqref{eq: metr_sub0} to get
\begin{align}
 \frac{1}{2} \| x^k - &\tilde{x}^k_\star \|^2_{\tau^{-1}} + \frac{1}{2} \| y^k - y^k_\star \|^2_{\Dsigma^{-1}} \notag \\
&\leq (2+ 2\eta^2 \| N-H\|^2) \mathbb{E}_k \left[ \frac{1}{2} \| x^{k+1} - x^k \|^2_{\tau^{-1}} + \frac{1}{2} \| y^{k+1} - y^k \|^2_{\Dsigma^{-1}P^{-1}} \right].\label{eq: metr_sub_last}
\end{align}
Herein, we denote $\zeta = 2 + 2\eta^2 \| \bs H-\bs M\|^2$.

By using~\eqref{eq: v_lw_bd}, we have that, for all $\alpha\in[0,1]$
\begin{align}
\mathbb{E}_{k-1} \left[ V(z^k - z^{k-1}) \right] &\geq C_1 \mathbb{E}_{k-1} \left[ \frac{1}{2} \|x^k -x^{k-1}\|^2_{\tau^{-1}} + \frac{1}{2} \| y^k - y^{k-1}\|^2_{\Dsigma^{-1}P^{-1}} \right] \notag \\
&\geq \frac{C_1}{\zeta} \left( \frac{\alpha}{2} \| x^{k-1} - \tilde{x}^{k-1}_\star \|^2_{\tau^{-1}} + \frac{1}{2} \| y^{k-1} - y^{k-1}_\star \|^2_{\Dsigma^{-1}} \right), \label{eq: metr_sub}
\end{align}
where the second inequality is due to~\eqref{eq: metr_sub_last} and $\alpha \geq 1$.

We have, by definition of $x_{k-1}^{\star}$ that
\begin{align*}
\Delta^{k-1} &\leq V_k(x^{k-1}-\tilde x^{k-1}_\star, y^k - y^k_\star) \notag\\
&=  \frac{1}{2} \| x^{k-1} - \tilde x^{k-1}_\star\|^2_{\tau^{-1}} + \frac{1}{2} \| y^{k}-y^k_\star\|^2_{\Dsigma^{-1}P^{-1}} + \frac{1}{2} \| y^{k}-y^{k-1}\|^2_{\Dsigma^{-1}P^{-1}} \notag \\
&\qquad -  \langle P^{-1} A(x^{k-1}-\tilde x^{k-1}_\star), y^{k}-y^{k-1} \rangle.
\end{align*} 
Next, by Cauchy-Schwarz and Young's inequalities with~\eqref{eq: ss_rules}, we have
\begin{equation*}
-  \langle P^{-1} A(x^{k-1}-\tilde x^{k-1}_\star), y^{k}-y^{k-1} \rangle \leq \frac{\gamma}{2} \|y^k-y^{k-1}\|^2_{\Dsigma^{-1}P^{-1}} + \frac{\gamma}{2} \| x^{k-1}-\tilde x^{k-1}_\star\|^2_{\tau^{-1}}.
\end{equation*}
Using the final estimate and adding and subtracting $\frac{1+\gamma}{2\alpha}\|y^{k-1}- y^{k-1}_\star\|^2_{\Dsigma^{-1}}$ gives
\begin{multline}
\Delta^{k-1} \leq \frac{1+\gamma}{2} \| x^{k-1}-\tilde x^{k-1}_\star\|^2_{\tau^{-1}} + \frac{1+\gamma}{2\alpha}\|y^{k-1}- y^{k-1}_\star\|^2_{\Dsigma^{-1}} \\
+ \frac{1+\gamma}{2} \|y^k-y^{k-1}\|^2_{\Dsigma^{-1}P^{-1}} - \frac{1+\gamma}{2\alpha}\|y^{k-1}- y^{k-1}_\star\|^2_{\Dsigma^{-1}}.
\end{multline}
We now take conditional expectation of both sides and use~\eqref{eq: metr_sub} to get
\begin{multline*}
\mathbb E_{k-1}\left[\Delta^{k-1}\right]  \leq \frac{(1+\gamma)\zeta}{C_1 \alpha} \mathbb E_{k-1} \left[ V(z^k - z^{k-1})\right] + \frac{1+\gamma}{2} \mathbb E_{k-1} \left[\| y^{k} - y^{k-1}\|^2_{\Dsigma^{-1}P^{-1}}\right] \notag \\
+ \frac{1}{2} \mathbb E_{k-1} \left[\| y^k - y^k_\star\|^2_{\Dsigma^{-1}P^{-1}}\right]- {\frac{1+\gamma }{2\alpha}} \| y^{k-1} - y^{k-1}_\star\|_{\Dsigma^{-1}}^2.
\end{multline*}
By using~\eqref{eq: v_lw_bd} and requiring that $\frac{(1+\gamma)}{C_1} \leq \frac{(1+\gamma)\zeta}{C_1\alpha}$, or equivalently $\zeta \geq \alpha$, which is not restrictive since $\alpha$ is finite, and one can increase $\eta$ as in~\eqref{eq: def_ms} to satisfy the requirement, we can combine the first two terms in the right hand side to get
\begin{align*}
\mathbb E_{k-1} \left[\Delta^{k-1}\right] \leq \frac{2(1+\gamma)\zeta}{C_1 \alpha} \mathbb E_{k-1} \left[V(z^k - z^{k-1})\right] &+ \frac{1}{2} \mathbb E_{k-1} \left[\| y^k - y^k_\star\|^2_{\Dsigma^{-1}P^{-1}} \right] \notag \\
&- {\frac{1+\gamma }{2\alpha}} \| y^{k-1} - y^{k-1}_\star\|_{\Dsigma^{-1}}^2.
\end{align*}
We now insert this inequality into~\eqref{eq: fejer_type} and use that $\mathbb{E}_{k-1} \left[ \mathbb{E}_k \left[ \Delta^k \right] \right] =\mathbb{E}_{k-1} \left[ \Delta^k \right]$
\begin{align*}
\mathbb{E}_{k-1} \left[ \Delta^k \right] &\leq \mathbb{E}_{k-1} \left[ \Delta^{k-1} \right]  - \frac{C_1\alpha}{2(1+\gamma)\zeta} \mathbb{E}_{k-1} \left[ \Delta^{k-1} \right] \notag \\
&+ \frac{C_1\alpha}{4(1+\gamma)\zeta} \mathbb{E}_{k-1}\left[\| y^k - y^k_\star \|^2_{\Dsigma^{-1}P^{-1}}\right]- \frac{C_1}{4\zeta} \| y^{k-1}-y^{k-1}_\star \|^2_{\Dsigma^{-1}} .
\end{align*}
We take full expectation and rearrange to get
\begin{multline}\label{eq: lin_rec1}
\mathbb{E}\Big[\Delta^k - \frac{C_1\alpha}{4(1+\gamma)\zeta}\|y^k -y^k_\star \|^2_{\Dsigma^{-1}P^{-1}} \Big]  \\
\leq \left(1-\frac{C_1\alpha}{2(1+\gamma)\zeta}\right) \mathbb{E}\Big[\Delta^{k-1} - \frac{C_1}{4\zeta(1-\frac{C_1\alpha}{2(1+\gamma)\zeta})} \| y^{k-1} - y^{k-1}_\star \|^2_{\Dsigma^{-1}}\Big].
\end{multline}
We require
\begin{equation}\label{eq: alpha_req1}
C_2 = \frac{C_1\alpha }{4\underline{p}(1+\gamma)\zeta} \leq \frac{C_1}{4\zeta} \leq \frac{C_1}{4\zeta(1-\frac{C_1\alpha}{2(1+\gamma)\zeta})} \iff \alpha \leq (1+\gamma)\underline{p}.
\end{equation}
Let us pick $\alpha = (1+\gamma)\underline{p}$ so that $C_2 = \frac{C_1}{4\zeta}$ and define
\begin{equation*}
\Phi^k = \Delta^k - C_2 \|y^k - y^k_\star\|^2_{\Dsigma^{-1}}.
\end{equation*}
We note~\eqref{eq: metr_sub} and~\eqref{eq: fejer_type} to have
\begin{equation*}
\| y^k - y^k_\star \|^2_{\Dsigma^{-1}} \leq \frac{2\zeta}{C_1} \mathbb{E}_k \left[ V(z^{k+1} - z^k)\right] \leq \frac{2\zeta}{C_1} \mathbb{E}_k\left[\Delta^k\right].
\end{equation*}
Then, we can lower bound $\Phi^k$ as
\begin{equation}\label{eq: phi_lb}
\mathbb{E}\left[\Phi^k\right] \geq \left( 1-C_2\frac{2\zeta}{C_1}\right)\mathbb{E}\left[ \Delta^k\right] = \frac{1}{2} \mathbb{E}\left[\Delta^k\right].
\end{equation}
Therefore, it follows that $\mathbb{E} \left[ \Phi^k \right]$ is nonnegative, by the definition of $\Delta^k$ and~\eqref{eq: vk_lw_bd}.

We can now rewrite~\eqref{eq: lin_rec1} as
\begin{equation*}
\mathbb{E}\left[\Phi^k\right] \leq (1-\rho)\mathbb{E}\left[\Phi^{k-1}\right],
\end{equation*}
where $\rho = \frac{C_1 \underline p}{2\zeta}$.
We have shown that $\Phi^k$ converges linearly to $0$ in expectation.

By~\eqref{eq: phi_lb}, it immediately follows that $\Delta^k$ converges linearly to $0$.

To conclude, we note $\Delta^k = V_{k+1}(x^k - x^k_\star, y^{k+1}-y^{k+1}_\star)$, and~\eqref{eq: vk_lw_bd}, from which we conclude linear convergence of $\| x^k - x^k_\star \|^2_{\tau^{-1}}$ and $\| y^{k+1} - y^{k+1}_\star \|^2_{\Dsigma^{-1}P^{-1}}$.

It is obvious to see that $0 < \rho $ follows by the fact that $\eta$ is finite by metric subregularity and $\rho < 1$ follows since $\gamma < 1$ and $\underline p \leq 1$.
\end{proof}

One important remark about~\Cref{eq: thm_lin} is that the knowledge of the metric subregularity constant $\eta$ is not needed for running the algorithm.
Step sizes are chosen as~\eqref{eq: ss_rules} and linear convergence follows directly when~\Cref{as: ms} holds.
Important examples where~\Cref{as: ms} holds are given in Section~\ref{sec: ms}.

Even though~\Cref{as: ms} is more general than prior assumptions for linear convergence and our result is agnostic to the choice of the step size, we observe in practice that SPDHG can be much faster than the rate derived in~\Cref{eq: thm_lin}.
We reflect on this issue more in Section~\ref{sec: conclusion} and present some open questions in this context.

\begin{remark}\label{rem: metric_sub}
Strictly speaking, metric subregularity is used in Theorem~\ref{eq: thm_lin} in the weighted norm, \emph{i.e.,}
\begin{equation*}
\dist_S(z, \mathcal{Z}^\star) \leq \eta \dist_S(0, Fz),
\end{equation*}
where $S = \diag(\tau^{-1}1_p, \sigma_1^{-1}, \dots, \sigma_n^{-1})$.
In terms of the definition in~\eqref{eq: def_ms} if $\eta_0$ is the constant using the standard Euclidean norm, it is obvious that $\eta \leq \| S\|\| S^{-1}\| \eta_0$, but we use $\eta$ in~\Cref{eq: thm_lin} since it can be smaller, resulting in a better rate.
\end{remark}

\subsection{Sublinear convergence}
In this section, we prove $\mathcal{O}(1/k)$ convergence rates for the ergodic sequence with different optimality measures. 
\subsubsection{Convergence of expected primal-dual gap}\label{sec: pd_gap}
We recall the definition of the primal-dual gap function,
\begin{align}
G(\bar{x}, \bar{y}) &= \sup_{z\in\mathcal{Z}} \mathcal{H}(\bar{x}, \bar{y}; x, y) \notag \\
&:= \sup_{z\in\mathcal{Z}} g(\bar x) + \langle A \bar x, y \rangle - f^\ast(y) - g(x) - \langle Ax, \bar y \rangle + f^\ast(\bar y).\label{eq: h_def}
\end{align}

It is also possible to consider the restricted primal-dual gap in the sense of~\cite{chambolle2018stochastic, chambolle2011first}, which for any set $\mathcal{B} = \mathcal{B}_x \times \mathcal{B}_y \subseteq\mathcal{Z}$ would correspond to
\begin{equation}\label{eq: def_pdgap}
G_{\mathcal{B}}(\bar{x}, \bar{y}) = \sup_{z\in\mathcal{B}} \mathcal{H}(\bar{x}, \bar{y}; x, y).
\end{equation}
The main quantity of interest for randomized algorithms is the expected restricted primal-dual gap $\mathbb{E}\left[ G_{\mathcal{B}}(\bar x, \bar y) \right]$.
As also mentioned in~\cite{dang2014randomized}, showing convergence rate for this quantity is not straightforward, as the interplay of supremum and expectation can be problematic.
In~\cite{dang2014randomized}, convergence rate is shown in a weaker measure named as perturbed gap function.
We show in the sequel that obtaining the guarantee in expected primal-dual gap is also possible, however, with a more involved analysis.

The expected primal-dual gap proof in~\cite{chambolle2018stochastic} has a technical issue, near the end of the proof in~\cite[Theorem 4.3]{chambolle2018stochastic}. Since the supremum of expectation is upper bounded by the expectation of the supremum, which is in the definition of expected primal-dual gap~\eqref{eq: def_pdgap}, the order of expectation in the proof is incorrect. As we could not find a simple way of fixing the issue using the existing techniques, we introduce a new technique and provide a proof to show that the conclusions of~\cite[Theorem 4.3]{chambolle2018stochastic}, for the primal-dual gap, are still correct, with different constants in the bound.

Our technique in the following proof is inspired by the stochastic approximation literature of variational inequalities and saddle point problems~(see~\cite[Lemma 3.1, Lemma 6.1]{nemirovski2009robust} for a reference), where such an analysis is used to obtain $\mathcal{O}(1/\sqrt{k})$ rates. In the new proof, we adapt this idea by using the structure of primal-dual coordinate descent to obtain the optimal $\mathcal{O}(1/k)$ rate of convergence. Our technique uses the Euclidean structure of the dual update of SPDHG, therefore might not be directly applicable to cases where general Bregman distances are used for proximal operator, such as in~\cite{lan2018optimal,lan2018random}.

We start with a lemma to decouple supremum and expectation in the proof.
\begin{lemma}\label{lem: new_erg_decoupling}
Given a point $\tilde{y}^1 \in \mathcal{Y}$, for $k \geq 1$, we define the sequences
\begin{equation}\label{eq: def_vk}
v^{k+1} = y^k - \hat{y}^{k+1} - P^{-1}(y^{k} - y^{k+1}),~~ \text{ and, }~~ \tilde{y}^{k+1} = \tilde{y}^{k} - P v^{k+1}.
\end{equation}
Then, we have for any $y\in\mathcal{Y}$,
\begin{align}
\sum_{k=1}^K \langle \tilde{y}^{k} - y, v^{k+1} \rangle_{\Dsigma^{-1}} &\leq \frac{1}{2} \| \tilde{y}^1 - y \|^2_{\Dsigma^{-1}P^{-1}} + \sum_{k=1}^K \frac{1}{2} \| v^{k+1} \|^2_{\Dsigma^{-1}P},\label{eq: yt_bd}\\
\mathbb{E}\left[  \sum_{k=1}^K \frac{1}{2} \| v^{k+1} \|^2_{\Dsigma^{-1}P}\right] &\leq \frac{1}{C_1} \Delta^0. \label{eq: sumvk_bd}
\end{align}
Moreover, $v^k$ and $\tilde{y}^{k}$ are $\mathcal{F}_k$-measurable and $\mathbb{E}_k \left[ v^{k+1} \right] = 0$.
\end{lemma}
\begin{proof}
For brevity in this proof, we denote $\Upsilon = \Dsigma^{-1}P^{-1}$. 
We have $\forall y\in\mathcal{Y}$,
\begin{align}
\frac{1}{2}\| \tilde{y}^{k+1} - y \|^2_\Upsilon &= \frac{1}{2} \| \tilde{y}^{k} - y \|^2_\Upsilon - \langle P v^{k+1}, \tilde{y}^{k}-y\rangle_\Upsilon + \frac{1}{2} \| Pv^{k+1} \|^2_\Upsilon\notag \\
&=\frac{1}{2} \| \tilde{y}^{k} - y \|^2_{\Dsigma^{-1}P^{-1}} - \langle v^{k+1}, \tilde{y}^{k}-y\rangle_{\Dsigma^{-1}} + \frac{1}{2} \| v^{k+1} \|^2_{\Dsigma^{-1}P}.\notag
\end{align}
Summing this equality gives the first result. 

For the second result, we use $\mathbb{E}_k \left[ P^{-1}(y^{k} - y^{k+1}) \right] = y^k - \hat{y}^{k+1}$, tower property, and the definition of variance,
\begin{align*}
\mathbb{E}\left[  \sum_{k=1}^K \frac{1}{2} \| v^{k+1} \|^2_{\Dsigma^{-1}P}\right] &= \sum_{k=1}^K  \frac{1}{2}\mathbb{E}\left[ \mathbb{E}_k \left[ \| v^{k+1} \|^2_{\Dsigma^{-1}P} \right]\right] \notag \\
&\leq \sum_{k=1}^K  \frac{1}{2}\mathbb{E}\left[ \mathbb{E}_k \left[ \| P^{-1}(y^{k+1}-y^k) \|^2_{\Dsigma^{-1}P} \right]\right] \notag\\
&=\sum_{k=1}^K \frac{1}{2}\mathbb{E}\left[ \|y^{k+1}-y^k \|^2_{\Dsigma^{-1}P^{-1}} \right] \leq \frac{1}{C_1} \Delta^0, \notag
\end{align*}
where the last inequality follows by $\sum_{k=1}^{\infty}\mathbb{E}\left[V(z^{k+1}-z^k)\right] \leq \Delta^0$ from~\Cref{thm: conv_as} and
$\frac{1}{2} \| y^{k+1} - y^k \|^2_{\Dsigma^{-1}P^{-1}} \leq \frac{1}{C_1} V(z^{k+1}-z^k)$ from~\Cref{lem: one_iteration}.

Other results follow immediately by the definition of the sequences and the equality $\mathbb{E}_k \left[ y^{k+1} - y^k \right] = P(\hat{y}^{k+1}-y^k)$.
\end{proof}

A direct proof of~\Cref{lem: one_iteration} would proceed by developing terms involving random quantities, by utilizing conditional expectations (see~\cite{chambolle2018stochastic}).
In this case, however, our approach is to proceed without using conditional expectation since the quantity of interest requires us to take first supremum and then the expectation of the estimates.
Our proof strategy is to characterize the error term, and then utilize the results~\Cref{lem: new_erg_decoupling} to decouple and bound this term.
First, we give the variant of~\Cref{lem: one_iteration} without taking expectations, with its proof given in~Section~\ref{sec: lemma_for_erg}.
\begin{lemma}\label{eq: lem_ergo_noexp}
We define $ f^\ast_{P}(y) = \sum_{i=1}^n p_i f^\ast_i(y_i)$,
and similar to~\eqref{eq: bregman_f} $D_{f^\ast}^{P}(\bar y; z) = \sum_{i=1}^n p_i f^\ast_i(\bar y_i) - p_if^\ast_i(y_i) - \langle (Ax)_i, p_i(\bar y - y)_i\rangle$ and recall the definitions of $V$ and $V_k$ from
\Cref{lem: one_iteration} and $\mathcal{H}$ from~\eqref{eq: h_def}.
Then, it holds that
\begin{align}\label{eq: lem49first}
\mathcal{H}(x^k, y^{k+1}; x, y) &\leq V_k(x^{k-1} - x, y^k - y) - V_{k+1}(x^k - x, y^{k+1} - y) - V(z^k - z^{k-1}) \notag \\
&+ \mathcal{E}^k+ D_{f^\ast}^{P^{-1}-I}(y^k; z) - D_{f^\ast}^{P^{-1}-I}(y^{k+1}; z) - \langle y, v^{k+1} \rangle_{\sigma^{-1}},
\end{align}
where $v^{k+1} = y^k - \hat{y}^{k+1} - P^{-1}(y^{k} - y^{k+1})$ and
\begin{align}
\mathcal{E}^k &= \frac{1}{2} \big[ \| y^k \|^2_{\sigma^{-1}} - \| \hat{y}^{k+1} \|^2_{\sigma^{-1}} - \big(\|y^k\|^2_{\sigma^{-1}P^{-1}} - \| y^{k+1} \|^2_{\sigma^{-1}P^{-1}} \big) \big]\notag \\
&+\frac{1}{2} \| y^{k+1} - y^{k}\|^2_{\sigma^{-1}P^{-1}} - \frac{1}{2} \| \hat{y}^{k+1}-y^k\|^2_{\sigma^{-1}} +f^\ast(y^k) - f^\ast(\hat{y}^{k+1})\notag\\
& -(f^\ast_{P^{-1}}(y^k) - f^\ast_{P^{-1}}(y^{k+1})) -\langle Ax^k, y^k - \hat{y}^{k+1} - P^{-1}(y^k - y^{k+1}) \rangle,\label{eq: e_def}
\end{align}
and also $\mathbb{E}_k \left[ \mathcal{E}^k \right] = 0$.
\end{lemma}

With this lemma, we identify the problematic inner product for deriving the rate for expected gap: $\langle y, v^{k+1} \rangle$~(see \eqref{eq: lem49first}). 
This is the only term coupling the free variable $z$ and random term $v^{k+1}$.
In the next theorem, we use~\Cref{lem: new_erg_decoupling} to manipulate this inner product. For the rest, we can observe in~\eqref{eq: lem49first} that the terms with $V_k$ telescopes, $\mathcal{E}^k$ has expectation $0$ and it is independent of free variable $z$.
\begin{theorem}\label{th: new_gap_th}
Let~\Cref{as: asmp1} hold. Define the sequences $x^K_{av} = \frac{1}{K}\sum_{k=1}^K x^k$ and $y^{K+1}_{av} = \frac{1}{K}\sum_{k=1}^K y^{k+1}$.
Then, for any set $ \mathcal{B} =\mathcal{B}_x \times \mathcal{B}_y\subseteq \mathcal{Z}$, the following result holds for the expected restricted primal dual gap defined in~\eqref{eq: h_def}
\begin{align}
\mathbb{E}\left[ \sup_{z \in \mathcal{B} } \mathcal{H}(x^K_{av}, y^{K+1}_{av}; x, y) \right] = \mathbb{E}\left[ G_{\mathcal{B}}(x^K_{av}, y^{K+1}_{av}) \right] \leq \frac{C_\mathcal{B}}{K},\label{eq: pd_gap}
\end{align}
where
\begin{multline*}
C_{\mathcal{B}} = \frac{1+2c}{2}\sup_{x\in\mathcal{B}_x} \| x^0 - x \|^2_{\tau^{-1}} + \sup_{y\in\mathcal{B}_y} \|{y}^1-y\|^2_{\sigma^{-1}P^{-1}} + f^\ast_{P^{-1}-I}(y^1) - f^\ast_{P^{-1}-I}(y^\star) \\
+\left(\frac{1}{C_1}+2c + c_1\right) \Delta^0
+ c\|x^0\|^2_{\tau^{-1}}+c \| y^1 -y^\star\|^2_{\sigma^{-1}P^{-1}} 
  + \frac{\|\sigma^{1/2}A\tau^{1/2}\|^2}{2c_1\underline p} \|x^\star \|^2_{\tau^{-1}},
\end{multline*}
where $c_1 = \|\tau^{1/2}A^\top\sigma^{1/2}P^{-1/2}\|$, $c=\| \tau^{1/2}A^\top(P^{-1}-I)\sigma^{1/2}P^{1/2} \|$, $C_1=1-\gamma$.
\end{theorem}

\begin{proof}
We start from the result of~\Cref{eq: lem_ergo_noexp}.
We have for the last term in~\eqref{eq: lem49first}
\begin{align}
-\langle y, v^{k+1} \rangle_{\Dsigma^{-1}} &= \langle \tilde{y}^{k} - y, v^{k+1} \rangle_{\Dsigma^{-1}} - \langle \tilde{y}^{k}, v^{k+1} \rangle_{\Dsigma^{-1}},\label{eq: new_erg_err1}
\end{align}
where $\tilde{y}^{k}$ is the random sequence defined in~\Cref{lem: new_erg_decoupling}.

We sum~\eqref{eq: lem49first} after using~\eqref{eq: new_erg_err1} and~\Cref{lem: one_iteration}
\begin{align}
\sum_{k=1}^K &\mathcal{H}(x^k, {y}^{k+1}; x, y) \leq -V_{K+1}(x^K - x, y^{K+1} - y) + V_1(x^0 - x, y^1 - y) \notag \\
&+D_{f^\ast}^{P^{-1}-I}(y^1; z) - D_{f^\ast}^{P^{-1}-I}(y^{K+1}; z) \notag \\
&+\sum_{k=1}^K \left(\langle \tilde{y}^{k} - y, v^{k+1} \rangle_{\Dsigma^{-1}} - \langle \tilde{y}^{k}, v^{k+1} \rangle_{\Dsigma^{-1}}+\mathcal{E}^k\right), \label{eq: main_ineq2}
\end{align}
Next, by Young's inequality (see also~\eqref{eq: pf_inner_product})
\begin{equation}\label{eq: new_erg_young1}
-\langle A(x-x^K), P^{-1}(y^{K+1} - y^K) \rangle \leq \frac{\gamma}{2} \| x-x^K \|^2_{\tau^{-1}}
+ \frac{\gamma}{2} \| y^{K+1} - y^K \|^2_{\Dsigma^{-1}P^{-1}}.
\end{equation}
On~\eqref{eq: main_ineq2}, we use~\eqref{eq: yt_bd} from~Lem.~\ref{lem: new_erg_decoupling} with $\tilde y^1 \!=\! y^1\!=\! y^0$,~\eqref{eq: new_erg_young1} with the definition of $V_{K+1}(x^K \!-\! x, y^{K+1} \!-\! y)$ from~Lem.~\ref{lem: one_iteration} (see also~\eqref{eq: vk_lw_bd},~\eqref{eq: pf_inner_product}), and $\gamma \!<\! 1$ from~\eqref{eq: ss_rules}
\begin{align}
&\sum_{k=1}^K \mathcal{H}(x^k, {y}^{k+1}; x, y) \leq  \frac{1}{2} \| x^0 - x \|^2_{\tau^{-1}} + \| y^1 - y\|^2_{\Dsigma^{-1}P^{-1}}\notag \\
&+f^\ast_{P^{-1}-I}(y^1) - f^\ast_{P^{-1}-I}(y^{K+1}) + \langle Ax, (P^{-1}-I)(y^{K+1} - y^1) \rangle \notag\\
&+ \sum_{k=1}^K \Big(\frac{1}{2} \| v^{k+1} \|^2_{\Dsigma^{-1}P} -  \langle \tilde{y}^{k}, v^{k+1} \rangle_{\Dsigma^{-1}} +\mathcal{E}^k \Big).\label{eq: main_3}
\end{align}
We have $\langle Ax, (P^{-1}-I)(y^{K+1} - y^1) \rangle \leq c \left(\frac{1}{2} \| x \|^2_{\tau^{-1}} + \frac{1}{2} \| y^{K+1} - y^1 \|^2_{\Dsigma^{-1}P^{-1}}\right)$, where $c=\| \tau^{1/2}A^\top(P^{-1}-I)\sigma^{1/2}P^{1/2} \|$ and $\frac{1}{2} \| x \|^2_{\tau^{-1}} \leq  \| x-x^0\|^2_{\tau^{-1}} +  \| x^0 \|^2_{\tau^{-1}}$.

We use these inequalities, arrange~\eqref{eq: main_3}, and divide both sides by $K$
\begin{align}
\frac{1}{K} \sum_{k=1}^K \mathcal{H}(x^k, y^{k+1}; x, y)  &\leq \frac{1}{K} \bigg\{\frac{1+2c}{2} \| x^0 - x \|^2_{\tau^{-1}} + \| y^1-y\|^2_{\Dsigma^{-1}P^{-1}}+c\| x^0\|^2_{\tau^{-1}} \notag\\
&+\frac{c }{2} \| y^{K+1}- y^1 \|^2_{\Dsigma^{-1}P^{-1}} +f^\ast_{P^{-1}-I}(y^1) - f^\ast_{P^{-1}-I}(y^{K+1}) \notag\\
&+ \sum_{k=1}^K \left(\frac{1}{2} \| v^{k+1} \|^2_{\Dsigma^{-1}P} - \langle \tilde{y}^{k}, v^{k+1} \rangle_{\Dsigma^{-1}}+ \mathcal{E}^k\right) \bigg\}. \label{eq: for_smgap}
\end{align}
We now take supremum of~\eqref{eq: for_smgap} with respect to $z$, note that only the first two terms on the right hand side depend on $z=(x, y)$, and $x^0$, ${y}^1$ are deterministic. 
Then we take expectation of both sides of~\eqref{eq: for_smgap} 
\begin{align}
\mathbb{E} \bigg[ \sup_{z\in\mathcal{B}} &\frac{1}{K} \sum_{k=1}^K \mathcal{H}(x^k, y^{k+1}; x, y) \bigg] \leq \frac{1}{K} \Bigg\{\sup_{z \in\mathcal{B}} \left\{\frac{1+2c}{2}\| x^0 - x \|^2_{\tau^{-1}} + \| y^1-y\|^2_{\Dsigma^{-1}P^{-1}}\right\} \notag\\
&+ \mathbb{E}\left[ \frac{c }{2}\| y^{K+1}- y^1 \|^2_{\Dsigma^{-1}P^{-1}}  + f^\ast_{P^{-1}-I}(y^1) - f^\ast_{P^{-1}-I}(y^{K+1}) \right] + c\| x^0 \|^2_{\tau^{-1}}\notag\\
&+ \sum_{k=1}^K \frac{1}{2} \mathbb{E}\left[ \| v^{k+1} \|^2_{\Dsigma^{-1}P}\right] - \sum_{k=1}^K \mathbb{E}\left[ \langle \tilde{y}^{k}, v^{k+1} \rangle_{\Dsigma^{-1}}\right] + \sum_{k=1}^K\mathbb{E}\left[\mathcal{E}^k \right]\Bigg\}. \label{eq: gap_last}
\end{align}
As $\tilde{y}^k$ is $\mathcal{F}_k$-measurable, $\mathbb{E}_k \left[v^{k+1}\right] = 0$, by~\Cref{lem: new_erg_decoupling}, and by the tower property,
\begin{align}
\mathbb{E} \left[\sum_{k=1}^K \langle \tilde{y}^{k}, v^{k+1} \rangle_{\Dsigma^{-1}}\right] &=  \sum_{k=1}^K \mathbb{E}\left[ \mathbb{E}_k \left[ \langle \tilde{y}^{k}, v^{k+1} \rangle_{\Dsigma^{-1}} \right]\right] \notag \\
&=  \sum_{k=1}^K \mathbb{E}\left[  \langle \tilde{y}^{k}, \mathbb{E}_k[ v^{k+1}] \rangle_{\Dsigma^{-1}} \right] = 0.\label{eq: new_erg_err_lin}
\end{align}
On~\eqref{eq: gap_last}, we use~\eqref{eq: sumvk_bd} from~\Cref{lem: new_erg_decoupling},~\eqref{eq: new_erg_err_lin} and \\
$\sum_{k=1}^K \mathbb{E}[\mathcal{E}^k] = \sum_{k=1}^K \mathbb{E}\left[\mathbb{E}_k\left[\mathcal{E}^k\right]\right] = 0$, which follows from~\Cref{eq: lem_ergo_noexp} along with the tower property, to obtain
\begin{align}\label{eq: new_erg_lastbd}
\mathbb{E} \bigg[ \sup_{z\in\mathcal{B}}\frac{1}{K} \sum_{k=1}^K \mathcal{H}(x^k, y^{k+1}; x, y) \bigg] &\leq \sup_{z\in\mathcal{B}}\left\{ \frac{1+2c}{2K} \| x^0 - x \|^2_{\tau^{-1}} + \frac{1}{K} \| {y}^1-y\|^2_{\Dsigma^{-1}P^{-1}} \right\} \notag\\
&+\frac{c }{2K} \mathbb{E}\left[ \| y^{K+1}- y^1 \|^2_{\Dsigma^{-1}P^{-1}} \right] + \frac{c}{K}\|x^0\|^2_{\tau^{-1}} \notag \\
&+\frac{1}{K}\mathbb{E}\left[ f^\ast_{P^{-1}-I}(y^1) - f^\ast_{P^{-1}-I}(y^{K+1}) \right] + \frac{1}{C_1 K} \Delta^0.
\end{align}
By~\Cref{thm: conv_as} and~\Cref{lem: one_iteration}, $\mathbb{E} \left[ \|y^{K+1} - y^\star \|^2_{\Dsigma^{-1}P^{-1}} \right] \leq 2\Delta^0$, and by Jensen's inequality, $\mathbb{E} \left[ \| y^{K+1}-y^\star \|_{\Dsigma^{-1}P^{-1}} \right] \leq \sqrt{2\Delta^0}$.
With these estimations we have
\begin{equation}\label{eq: erg_bd_yk_term}
\mathbb{E} \left[ \| y^{K+1}-y^1 \|^2_{\Dsigma^{-1}P^{-1}} \right] \leq 2 \| y^1-y^\star \|^2_{\Dsigma^{-1}P^{-1}} + 4\Delta^0.
\end{equation} 
As $f_i$ is proper, l.s.c., convex, and $A_i x^\star \in \partial f^\ast_i(y^\star_i)$, we additionally note that 
\begin{align*}
 f^\ast_i(y^{K+1}_i)  &\geq  f^\ast_i(y^\star_i) + \langle A_i x^\star, y^{K+1}_i - y^\star_i \rangle \notag \\
 &\geq f^\ast_i(y^\star_i) - \| A_i x^\star \|_{\sigma_i}  \| y^{K+1}_i - y^\star_i \|_{\Dsigma_i^{-1}},
\end{align*}
and by substitution, Young's inequality, and defining $c_1 = \|\tau^{1/2}A^\top\sigma^{1/2}P^{-1/2}\|$
\begin{align}
\mathbb{E} \big[ f^\ast_{P^{-1}-I}&(y^{K+1})\big] = \sum_{i=1}^n \left( \frac{1}{p_i} - 1 \right) \mathbb{E} \left[ f^\ast_i(y^{K+1}_i)\right] \notag \\
&\geq \sum_{i=1}^n \left( \frac{1}{p_i} - 1 \right) \Big(f^\ast_i(y^\star_i) - \frac{1}{2c_1}\| A_i x^\star \|^2_{\sigma_i} -\frac{c_1}{2}\mathbb{E}\left[ \| y^{K+1}_i - y^\star_i \|^2_{\Dsigma_i^{-1}} \right] \Big) \notag \\
&\geq \sum_{i=1}^n \left( \frac{1}{p_i} - 1 \right) f^\ast_i(y^\star_i) - \frac{1}{2c_1\underline p}\| A x^\star \|^2_{\sigma}-\frac{c_1}{2}\mathbb{E}\|y^{K+1}-y^\star \|^2_{\sigma^{-1}P^{-1}}.\label{eq: fconj_lbd}
\end{align}
We now use~\eqref{eq: erg_bd_yk_term} and~\eqref{eq: fconj_lbd} in~\eqref{eq: new_erg_lastbd}, use $\mathbb{E} \left[ \|y^{K+1} - y^\star \|^2_{\Dsigma^{-1}P^{-1}} \right] \leq 2\Delta^0$ and use definition of $c_1$ to obtain
\begin{align}
\mathbb{E} \Bigg[ \sup_{z\in\mathcal{B}}&\frac{1}{K} \sum_{k=1}^K \mathcal{H}(x^k, y^{k+1}; x, y) \Bigg] \leq \frac{1+2c}{2K}\sup_{x\in\mathcal{B}_x} \| x^0 - x \|^2_{\tau^{-1}} +\frac{1}{K} \sup_{y\in\mathcal{B}_y} \| y^1-y\|^2_{\Dsigma^{-1}P^{-1}} \notag\\
&+\frac{c}{K} \| y^1-y^\star \|^2_{\Dsigma^{-1}P^{-1}} + \frac{2c}{K}\Delta^0 + \frac{c}{K} \|x^0\|^2_{\tau^{-1}} + \frac{1}{K} \left(f^\ast_{P^{-1}-I}(y^1)-f^\ast_{P^{-1}-I}(y^\star) \right)\notag \\
&+ \frac{\|\sigma^{1/2}A\tau^{1/2}\|^2}{2c_1\underline p K} \|x^\star \|^2_{\tau^{-1}} + \frac{c_1}{K} \Delta^0 + \frac{1}{C_1K } \Delta^0 =: \frac{C_{\mathcal{B}}}{K}.\notag 
\end{align}
We define as $C_{\mathcal{B}}$ the constant of right hand side and use Jensen's inequality on the left hand side with definitions of $x^K_{av}$ and $y^{K+1}_{av}$ to get the result.
\end{proof}
\begin{remark}
In Thm.~\ref{th: new_gap_th}, when $p_i=\frac1 n$, setting scalar step sizes $\tau=\frac{1}{n\max_i\|A_i\|}$, $\sigma= \frac{1}{\max_i\|A_i\|}$ in view of~\eqref{eq: ss_rules} gives $\mathcal{O}(n[\|A\| + f^\ast(y^1)-f^\ast(y^\star)]\cdot\max_{z\in\mathcal{B}}[\|x\|^2+\|y\|^2])$ as the worst case order for $C_{\mathcal{B}}$.
\end{remark}

\subsubsection{Convergence of objective values}
\label{sec: ergodic_rate_sec}
The guarantee for the expected global primal-dual gap (see~\eqref{eq: pd_gap}) requires bounded primal and dual domains.

In this section, we show that $\mathcal{O}(1/k)$ rate of convergence in terms of objective values and/or feasibility can be shown with possibly unbounded primal and dual domains.
The case $f(\cdot) = \delta_b(\cdot)$ is studied in~\cite{luke2018block} and a similar result was derived.
The rate in~\cite{luke2018block} has a different nature in the sense that it is an almost sure rate where the constant depends on trajectory, whereas our rate is in expectation.
We use the smoothed gap function introduced in~\cite{tran2018smooth}, which,
for~\eqref{eq: prob_temp}, is defined as
\begin{equation}\label{eq: sm_gap_def}
\mathcal{G}_{\alpha, \beta}(x, y; \dot{x}, \dot{y}) = \sup_{u, v} g(x) + \langle Ax, v \rangle - f^\ast(v)
- g(u) - \langle Au, y \rangle + f^\ast(y) - \frac{\alpha}{2} \| u - \dot{x}\|^2 - \frac{\beta}{2}\| v-\dot{y}\|^2.
\end{equation}

\begin{theorem}\label{thm: thm_ergodic}
Let~\Cref{as: asmp1} hold.
We recall ${x}^K_{av} = \frac{1}{K} \sum_{k=1}^K x^k$.\\
$\bullet$ If $f$ is $L(f)$-Lipschitz continuous  and $y^1 \in \dom{f^\ast}$,
\begin{equation*}
\mathbb{E}\left[ f(A{x}^K_{av}) + g({x}^K_{av}) - f(Ax^\star) - g(x^\star) \right] \leq \frac{C_{e,1}}{{K}}.
\end{equation*}
$\bullet$ If $f(\cdot) = \delta_{\{b\}}(\cdot)$ with $b\in\mathcal{Y}$,
\begin{align*}
\mathbb{E} \left[  g(x^K_{av}) - g(x^\star) \right] \leq \frac{C_{e, 2}}{{K}},~~~ \mathbb{E} \left[ \|Ax^K_{av} - b\|_{\diag(\sigma)P} \right] \leq \frac{C_{e, 3}}{{K}},
\end{align*}
where (see~\Cref{th: new_gap_th})
\begin{multline*}
C_e = f^\ast_{P^{-1}-I}(y^1) - f^\ast_{P^{-1}-I}(y^\star)
+\left(\frac{1}{C_1}+2c + c_1\right) \Delta^0
\\ 
+ c\|x^0\|^2_{\tau^{-1}}+c \| y^1 -y^\star\|^2_{\sigma^{-1}P^{-1}}   + \frac{\|\sigma^{1/2}A\tau^{1/2}\|^2}{2c_1\underline p} \|x^\star \|^2_{\tau^{-1}},
\end{multline*}
$C_{e, 1} = C_e + \frac{2}{\underline{p}} L(f)^2 + \frac{1+2\gamma}{2} \| x^0 - x^\star \|^2_{\tau^{-1}}$, \\
$C_{e,3}=\frac{1}{2} \big\{ \| y^\star - {y}^1 \|_{\Dsigma^{-1}P^{-1}} +\big( \| y^\star - {y}^1 \|^2_{\Dsigma^{-1}P^{-1}} + 4C_e + 6 \| x^\star - x^0 \|_{\tau^{-1}}\big)^{1/2}\big\}$, \\
$C_{e, 2} = C_e + \frac{1}{2} \| y^\star - {y}^1 \|^2_{\Dsigma^{-1}P^{-1}} + \frac{1+2\gamma}{2} \| x^0 - x^\star \|^2_{\tau^{-1}} + \| y^\star \|_{\Dsigma^{-1}P^{-1}} C_{e,3}$.
\end{theorem}
\begin{proof}
For the smoothed gap (see~\eqref{eq: sm_gap_def}), from~\Cref{th: new_gap_th}, we have
\begin{align*}
\mathbb{E}\left[ \mathcal{G}_{\frac{1+2\gamma}{2K}, \frac{1}{2K}}(x^K_{av}, y^{K+1}_{av}; x^0, y^1) \right] \leq \frac{C_e}{K}.
\end{align*}
To see this, we proceed the same as in the proof of~\Cref{th: new_gap_th} until~\eqref{eq: for_smgap}. Then, we move the terms $\frac{1+2\gamma}{2K} \| x^0 - x \|^2_{\tau^{-1}}$ and $\frac{1}{K} \| {y}^1 - y \|^2_{\sigma^{-1}P^{-1}}$ to the left hand side, take supremum, use the definition of smoothed gap, then take expectations of both sides and use the same estimations as in the first part to conclude.

$\bullet$ When $f$ is Lipschitz continuous in the norm $\| \cdot\| _{\sigma}$, we will argue as in~\cite[Theorem 11]{fercoq2019coordinate}.
On~\eqref{eq: sm_gap_def}, with the parameters used in this theorem, we make the following observations.
By~\cite[Corollary 17.19]{bauschke2011convex}, when $f$ is $L(f)$-Lipschitz continuous in the norm $\| \cdot \| _{\diag(\sigma)}$, it follows that $\|y^1 - y \|_{\Dsigma^{-1}}^2 \leq 4
L(f)^2$.
By Lipschitzness and the definition of conjugate function, we can pick $y\in\partial f(Ax^K_{av}) \neq \emptyset$ such that $\langle Ax^K_{av}, y \rangle - f^\ast(y) = f(Ax^K_{av})$. Next by Fenchel-Young inequality, $f^\ast({y}^{K+1}_{av})-\langle A^\top {y}^{K+1}_{av}, x^\star\rangle \geq -f(Ax^\star)$.
We also use $\underline{p} = \min_i p_i$ to obtain (see~\eqref{eq: sm_gap_def})
\begin{align*}
\mathbb{E} \bigg[ \mathcal{G}_{\frac{1+2\gamma}{K},\frac{1}{K}}(x^{K}_{av}, {y}^{{K}+1}_{av}; &x^{0}, y^1) \bigg]  \geq \mathbb{E} \left[ f(Ax^K_{av}) + g(x^K_{av}) - f(Ax^\star) - g(x^\star)\right] \notag \\
&- \frac{2}{\underline{p} K} L(f)^2 - \frac{1+2\gamma}{2K} \| x^0-x^\star \|^2_{\tau^{-1}},
\end{align*}
where the result directly follows.

$\bullet$ When $f(\cdot) = \delta_b (\cdot)$, we use~\cite[Lemma 1]{tran2018smooth}, to obtain the bounds
\begin{align*}
&\mathbb{E} \left[ g(x^K_{av}) - g(x^\star)\right] \leq \mathbb{E} \left[ \mathcal{G}_{\frac{1+2\gamma}{2K},\frac{1}{2K}}(x^{K}_{av}, {y}^{{K}+1}_{av}; x^{0}, y^1)\right] \notag \\
&+ \frac{1+2\gamma}{2K} \| x^0 - x^\star \|^2_{\tau^{-1}}- \mathbb{E} \left[ \langle y^\star, Ax^K_{av} - b \rangle \right]+ \frac{1}{2K} \| y^\star - y^1 \|^2_{\Dsigma^{-1}P^{-1}}, \notag\\
&\mathbb{E} \left[ \| Ax^K_{av}-b\|_{\diag(\sigma)P} \right] \leq \frac{1}{2K} \bigg\{ \| y^\star - y^1 \|_{\Dsigma^{-1}P^{-1}}+ \bigg(\| y^\star - y^1 \|^2_{\Dsigma^{-1}P^{-1}}  \notag\\
&+4K \mathbb{E} \left[ \mathcal{G}_{\frac{1+2\gamma}{2K},\frac{1}{2K}}(x^{K}_{av}, {y}^{{K}+1}_{av}; x^{0}, y^1) \right] + 2(1+2\gamma) \| x^0 - x^\star \|^2_{\tau^{-1}} \bigg)^{1/2} \bigg\}.
\end{align*}

We use Cauchy-Schwarz inequality and the bound of 
$\mathbb{E}\left[ \| Ax_{av}^K - b\|_{\diag(\sigma)P} \right]$ on $\langle y^\star, Ax^K_{av} - b \rangle$  to conclude.
\end{proof}

\section{Related works}\label{sec: related_works}
We summarize the comparison of the most related primal-dual coordinate descent methods (PDCD) in~\Cref{tab:1} at Page \pageref{tab:2}.\\[2mm]
\textit{Primal camp.}
Stochastic gradient based methods (SGD) can be applied to solve
~\eqref{eq: prob_temp}
~\cite{robbins1951stochastic, nemirovski2009robust}.
SGD cannot get linear convergence except special cases~\cite{necoara2018randomized}.
Variance reduction based methods obtain linear convergence when the functions $f_i$ are smooth and $g$ is strongly convex; or $f_i$ are smooth and strongly convex~\cite{johnson2013accelerating, xiao2014proximal, allen2017katyusha}.
Smoothness of $f_i$ is equivalent to strong convexity of $f_i^\ast$.
Therefore, the linear convergence results of these methods require the similar assumptions as~\cite{chambolle2018stochastic}.
Moreover, as in~\cite{chambolle2018stochastic}, variance reduction based methods require knowing the constants $\mu_i$ and $\mu_g$ to set the algorithmic parameters accordingly, for obtaining linear convergence.

When $f_i(\cdot) =\delta_{\{{b}_i\}}(\cdot)$, SGD-type methods are proposed in~\cite{patrascu2017nonasymptotic, xu2018primal, fercoq2019almost}. However, these methods only obtain $\mathcal{O}(1/k)$ rate with strong convexity of $g$, since they focus on the general problem where the objective can be given in expectation form. Even though this rate is optimal for the given template, it is suboptimal for~\eqref{eq: prob_temp}.\\[2mm]
\textit{Primal-dual camp.}
This line of research uses coordinate descent type schemes for solving~\eqref{eq: prob_temp}.
Coordinate descent with random sampling for unconstrained optimization is proposed in~\cite{nesterov2012efficiency} and later generalized and improved in~\cite{richtarik2014iteration, fercoq2015accelerated}. These methods apply coordinate descent in the primal and obtain linear convergence rates with smooth and strongly convex $f_i$; or smooth $f_i$ and strongly convex $g$.

Another approach is to apply coordinate ascent in the dual to exploit separability of the dual in~\eqref{eq: prob_temp}. Stochastic dual coordinate ascent (SDCA) and its accelerated variant are proposed in~\cite{shalev2013stochastic, shalev2014accelerated}. These methods require smoothness of $f_i$ and strong convexity of $g$ for linear convergence and the strong convexity constants are used in the algorithms for setting the parameters.

The algorithm we analyzed in this paper is SPDHG, proposed in~\cite{chambolle2018stochastic}.
The authors proved linear convergence of the modified method SPDHG-$\mu$~\cite[Theorem 6.1]{chambolle2018stochastic} by assuming strong convexity of $f_i^\ast, g$ and special step sizes depending on  strong convexity constants. Iterate convergence and ergodic $\mathcal{O}(1/k)$ rate results in~\cite[Theorem 4.3]{chambolle2018stochastic} are given in terms of Bregman distances which is not a valid optimality measure in general. 
Our analysis for SPDHG shows linear convergence with standard step sizes in~\eqref{eq: ss_rules} and with weaker metric subregularity assumption (see Section~\ref{sec: ms}). Moreover, in the general convex case, we prove almost sure convergence of the iterates to a solution, which is stronger than Bregman distance based almost sure convergence in~\cite{chambolle2018stochastic}.
Finally, we prove $\mathcal{O}(1/k)$ rate for the ergodic sequence, with possibly unbounded domains, for optimality measures stronger than Bregman distances, such as expected primal-dual gap.
The comparison of the results is also summarized in~\Cref{tab:2}.

PDCD schemes similar to SPDHG are proposed in~\cite{zhang2017stochastic, dang2014randomized, fercoq2019coordinate}. These variants assume strong convexity of $f_i^\ast, g$ for linear rate of convergence. Only~\cite{fercoq2019coordinate} proved linear convergence with step sizes independent of strong convexity constants, to provide a partial answer for adaptivity of PDCD methods to strong convexity. However, as detailed in Table \ref{tab:1}, with dense $A$ matrix and uniform sampling, this method requires step sizes $n$ times smaller than~\eqref{eq: ss_rules} which can be problematic in practice (see~Section~\ref{sec: bp}). For sublinear convergence,~\cite{fercoq2019coordinate} proved $\mathcal{O}(1/\sqrt{k})$ rate on a randomly selected iterate, under similar assumption to ours whereas~\cite{zhang2017stochastic} requires boundedness of the dual domain, setting a horizon and proves primal-only rates.

PDCD algorithms are also studied in \cite{combettes2015stochastic, combettes2019stochastic, pesquet2015class}.
As mentioned in~\cite{fercoq2019coordinate,chambolle2018stochastic}, operator theory-based proofs of these methods require using step sizes depending on global constants about the problem, causing slow performance in practice.
PDCD methods for linearly constrained problems are studied in~\cite{alacaoglu2017smooth, dang2014randomized, luke2018block}, with sublinear rates.

Latafat et al.\ \cite{latafat2019new} proposed TriPD-BC and proved linear convergence for this method under metric subregularity.
There are two drawbacks of TriPD-BC for our setting.
First, when $A$ is not of special structure, such as block diagonal, one needs to use duplication for an efficient implementation (see~\cite{fercoq2019coordinate}).
Second issue is that as in~\cite{fercoq2019coordinate}, this method needs to use $n$ times smaller step sizes with dense $A$.
For the details of duplication and small step sizes, we refer to~\cite{fercoq2019coordinate}.
The need to use small step sizes seriously affects the practical performance of the algorithm (see Section~\ref{sec: bp}).

Some standard references for deterministic primal-dual algorithms are in~\cite{chambolle2011first,chambolle2016ergodic,he2012convergence,tran2018smooth,tran2018adaptive,esser2010general}.
As observed in~\cite{chambolle2018stochastic}, coordinate descent-based variants significantly increase the practical performance of these deterministic methods.

Our results imply global linear convergence for PDHG when $n=1$, answering the question posed in~\cite{chambolle2011first}: \textit{``It would be interesting to understand whether the steps can be estimated in Algorithm 1 without the a priori knowledge of $\mu_i, \mu_g$.''}
In the third part of~\Cref{as: ms}, compact domains are not needed for this case.
We highlight that such behaviour of deterministic primal-dual methods is investigated before in~\cite{liang2016convergence, latafat2019new}.\\[2mm]
\textit{Linear programming.}
A related notion to metric subregularity for linear programming is Hoffman's lemma due to classical result in~\cite{hoffman1952approximate}, which is used to show linear convergence of ADMM-type methods for LPs~\cite{yen2015sparse,yang2016linear,liu2018partial}. The drawback of these approaches is that the knowledge of the constant $\eta$ is required to run the algorithm, which is difficult to estimate. Our analysis recovers these results specific to LPs with a simpler algorithm that does not need the knowledge of $\eta$.

\section{Numerical evidence}
In this section, we support our theoretical findings by showing that SPDHG with step sizes in~\eqref{eq: ss_rules} obtains linear convergence for problems satisfying metric subregularity. 

The problems we solve in this section satisfy metric subregularity (see~Section~\ref{sec: ms}). 
However, among these problems, only ridge regression is strongly convex-strongly concave, thus this is the only problem where existing linear convergence results from~\cite{chambolle2018stochastic} apply by using the algorithm SPDHG-$\mu$~\cite[Theorem 6.1]{chambolle2018stochastic}. 
We show that even in this case, when strong convexity constants are small, applying SPDHG can be more beneficial for some datasets.
SPDHG-$\mu$ is not applicable for other problems due to lack of strong convexity or strong concavity.
We also illustrate favorable behavior of SPDHG against state-of-the-art methods SVRG~\cite{johnson2013accelerating}, accelerated SVRG~\cite{zhou2018simple} and PDCD algorithms using smaller step sizes with dense data, such as~\cite{fercoq2019coordinate}.

Due to limited space, we include results with one or two datasets for each problem.
For SPDHG, as suggested in~\cite{chambolle2018stochastic}, we use uniform sampling of coordinates and the step sizes $\tau=\frac{0.99}{n \max_i \|A_i\|}$ and $\sigma_i = \frac{0.99}{\| A_i \|}$ for all problems.
For the other methods, we use the suggested theoretical step sizes in the respective papers and we do not fine tune any of the methods.
\subsection{Sparse recovery with basis pursuit}\label{sec: bp}
Basis pursuit is a fundamental problem in signal processing~\cite{chen2001atomic} with applications in machine learning~\cite{goldstein2018phasemax, arora2018compressed}:
\begin{equation}
\min_{x\in\mathbb{R}^d} \| x\|_1: Ax=b.
\end{equation}
Since basis pursuit is PLQ, metric subregularity holds.
In this section, we aim to illustrate the importance of step sizes, as mentioned in Section~\ref{sec: related_works} and~Table~\ref{tab:1} and to verify linear convergence of SPDHG.
We compare SPDHG with coordinate descent version of Vu-Condat algorithm from~\cite{fercoq2019coordinate}, which we refer to as FB-VC-CD.
Since~\cite{latafat2019new} requires duplication for an efficient implementation for this problem, it uses the same step sizes as~\cite{fercoq2019coordinate}.
Thus, we only compare with FB-VC-CD and note that the practical performance of~\cite{latafat2019new} is expected to be similar to FB-VC-CD with same step sizes.

We generate the data matrix $A$ with $n=500$ and $d=1000$ and entries follow a standard normal distribution.
We generate a covariance matrix $\Sigma_{i,j} = \rho^{\vert i - j \vert}$ with $\rho = 0.5$ and a sparse solution $x^\star$ with $100$ nonzero entries.
We then compute $b=Ax^\star$.

\begin{figure}[h]\label{fig: bp}
\begin{center}
\includegraphics[scale=.35]{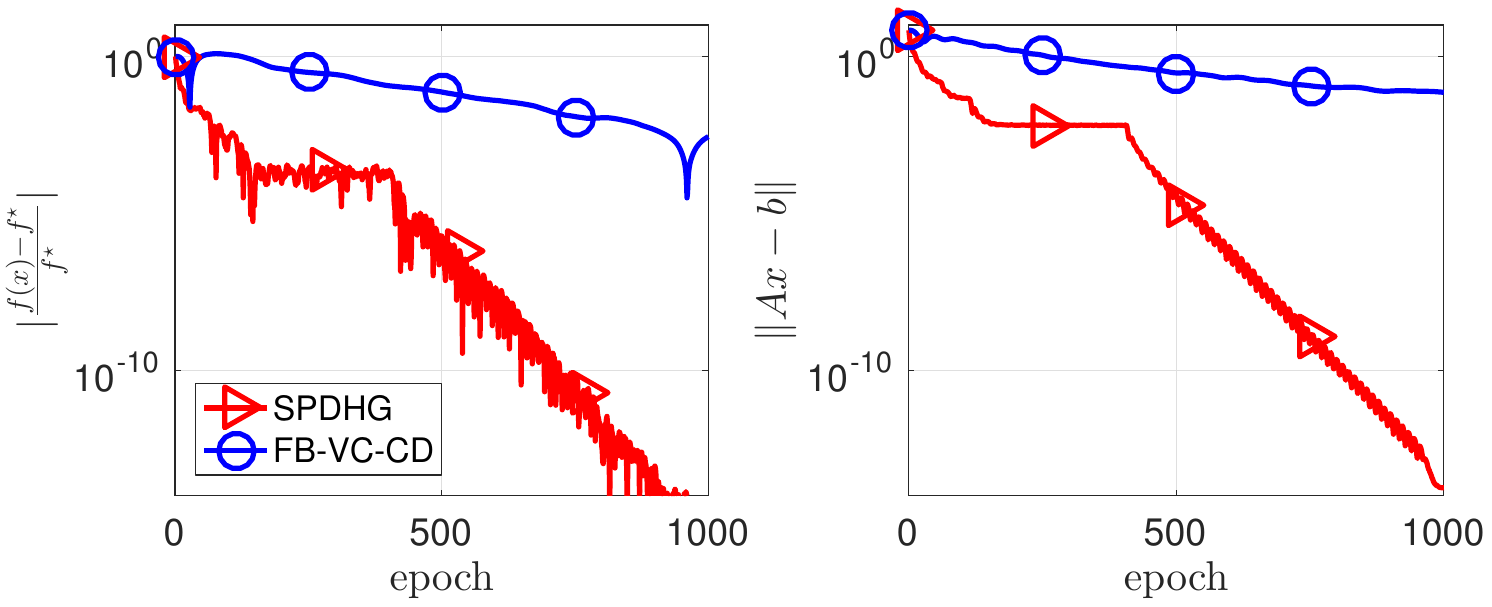}
\vspace{-.2cm}
\caption{Linear convergence of SPDHG for basis pursuit problem.}
\end{center}
\vskip -.2cm
\end{figure}

The analysis of SPDHG by~\cite{chambolle2018stochastic} shows $\mathcal{O}\left(1/k\right)$ rate on the Bregman distance to solution on the ergodic sequence whereas our analysis proves linear convergence on the last iterate.
On the other hand, FB-VC-CD is proven to have $\mathcal{O}\left( 1/\sqrt{k} \right)$ rate for this problem~\cite{fercoq2019coordinate}.
FB-VC-CD is tailored specially to exploit sparsity in the data.
However, the data is dense in this problem, which causes FB-VC-CD to use $n$ times smaller step sizes as shown in Figure~\ref{fig: bp}.
Because of this reason, FB-VC-CD exhibits a slow rate whereas SPDHG gets fast rate as predicted by our theoretical results.
\subsection{Lasso and ridge regression}
In this section we solve ridge regression and Lasso problems, formulated as
\begin{equation}\label{eq: ridge}
\min_{x\in\mathbb{R}^d} \frac{1}{2} \| Ax-b\|^2 + \frac{\lambda}{2} \| x\|^2, \text{ and, }
\min_{x\in\mathbb{R}^d} \frac{1}{2} \| Ax-b\|^2 + \lambda \| x\|_1,
\end{equation}
respectively.
In terms of structure,~\eqref{eq: ridge} is smooth and strongly convex, or equivalently, its Lagrangian is strongly convex-strongly concave.
For this problem class,~\cite{chambolle2018stochastic} showed linear convergence for the method SPDHG-$\mu$, which is a modified version of SPDHG using strong convexity and strong concavity constants for step sizes. 
In addition, SVRG and accelerated SVRG have linear convergence for this problem~\cite{xiao2014proximal, zhou2018simple, allen2017katyusha}.

\begin{figure}[h]\label{fig: ridge1}
\begin{center}
\includegraphics[scale=0.36]{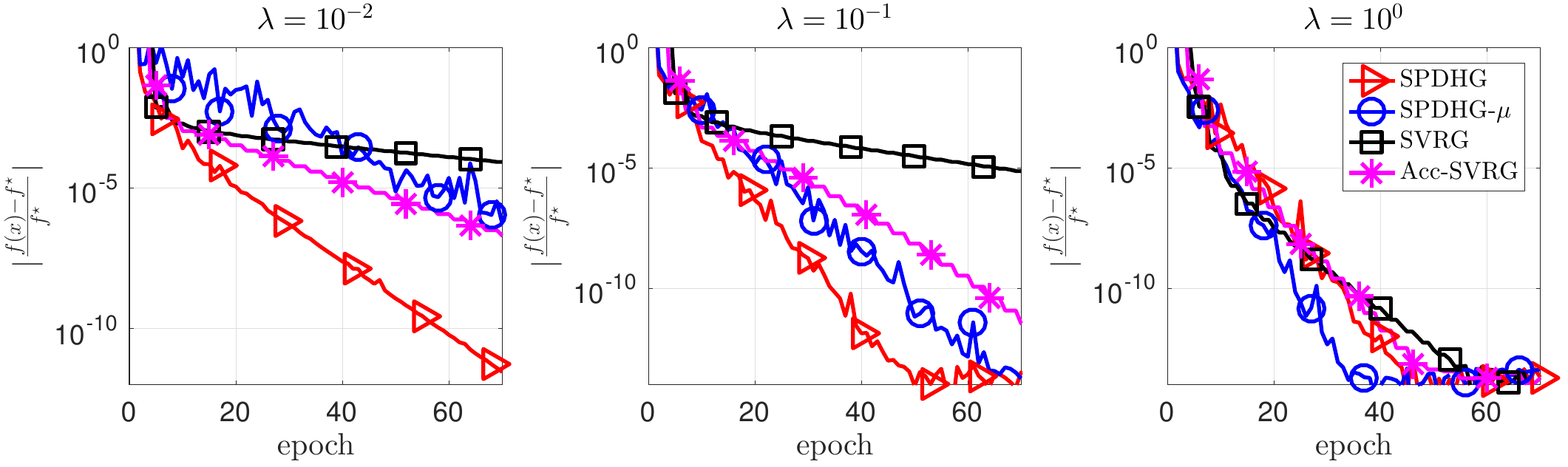}
\includegraphics[scale=0.36]{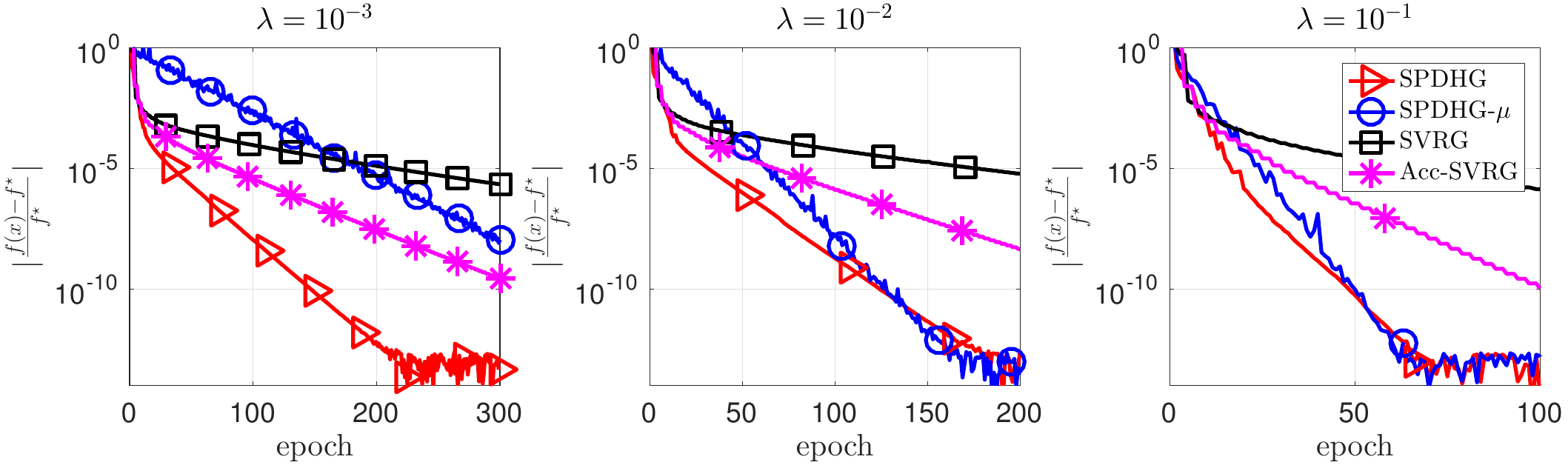}
\includegraphics[scale=0.36]{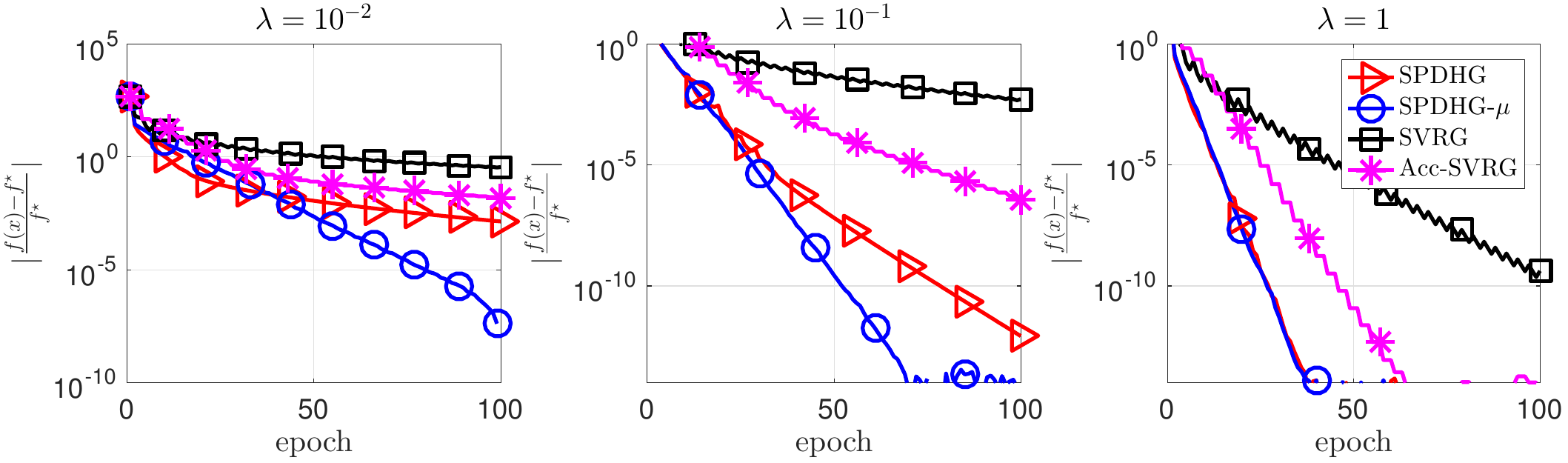}
\vspace{-0.2cm}
\caption{Ridge regression, first row: w8a, $n=49,749, d=300$; second row: sector, $n=6,412$, $d=55,197$; third row: YearPredictionMSD, $n=463,715, d=90$.}
\end{center}
\end{figure}
We use regression datasets from libsvm~\cite{CC01a}, perform row normalization, and use three different regularization parameters for each case. We compile the results in~\Cref{fig: ridge1} along with information on datasets and regularization parameters.

The aim in this experiment is not to argue that SPDHG gets the best performance in all cases since this is a very specific instance where most algorithms can get linear convergence.
Our goal is rather to show that even though our linear convergence results apply to a broad class of problems and SPDHG can apply to more general problems, it can still be competitive when compared to methods which are designed to exploit the structure of this specific setting.

When $n\geq d$, in~\Cref{fig: ridge1}, we see that for large regularization parameters, or equivalently, large strong convexity constants, SPDHG-$\mu$ is faster than SPDHG.
This is expected since SPDHG-$\mu$ is designed to use strong convexity as good as possible, whereas our result holds generically without any modifications on the algorithm.
Next, when strong convexity constant is small, SPDHG gets a faster linear rate than SPDHG-$\mu$, which suggests robustness of SPDHG over SPDHG-$\mu$ in this regime.
SPDHG also shows a more favorable performance than SVRG and accelerated SVRG.

When $n \leq d$, in~\Cref{fig: ridge1}, we see that SPDHG-$\mu$ shows faster convergence with small $\mu$. This seems intuitive, since in this case the strong convexity \emph{purely} comes from the regularization term. In this case, SPDHG-$\mu$ directly exploits this knowledge and shows a better performance.

We then solve Lasso~\eqref{eq: ridge}, for which SPDHG-$\mu$ does not apply and accelerated SVRG cannot get linear rates in general.
We compare with SVRG for varying regularization parameters, datasets with $n\leq d$ and $n\geq d$, and compile the results in~\Cref{fig: lasso}.
We observe that SPDHG converges linearly for this problem and exhibits a better practical performance than SVRG.
\begin{figure}[h]\label{fig: lasso}
\begin{center}
\includegraphics[scale=0.36]{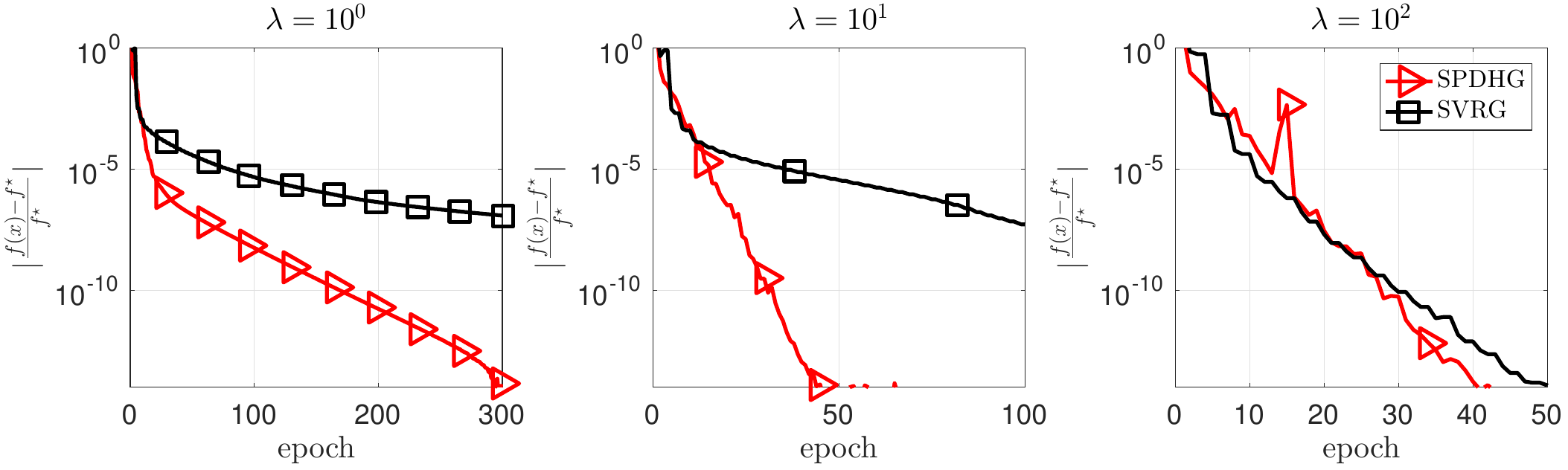}
\includegraphics[scale=0.36]{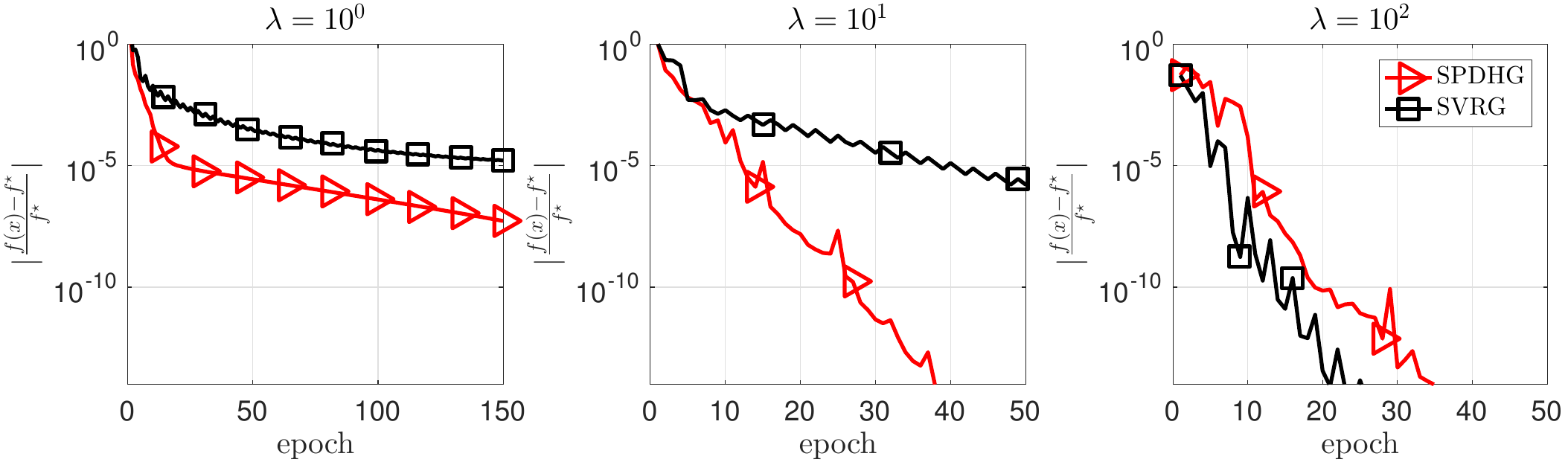}
\vspace{-0.2cm}
\caption{Lasso, top: mnist scale, $n=60,000, d=780$; bottom: rcv1.binary, $n=20,242$, $d=47,236$}
\end{center}
\vskip -0.3cm
\end{figure}
\section{Conclusions and open questions}\label{sec: conclusion}
In this section, we focus on the theory-practice gap mentioned in~Section~\ref{sec: lin_conv}, before~\Cref{rem: metric_sub}.
In particular, the main aim of~Section~\ref{sec: lin_conv} was to show that SPDHG obtains linear rate of convergence under general assumptions that hold for a large body of problems, with an agnostic step size selection.
A natural question is: How does this rate translate to practice?
For this purpose, we perform a controlled experiment on a simple problem
\begin{equation*}
\min_{x\in\mathbb{R}^d} \frac{\mu}{2} \| x \|^2: Ax=b,
\end{equation*}
with $d=n=10$.
After writing the KKT conditions, we obtain $F = \begin{bmatrix} \mu I & A^\top \\ A & 0 \end{bmatrix}$ and metric subregularity constant $\eta$ is the smallest eigenvalue of $F$ in absolute value.

For simplicity, we run PDHG, which is a specific case of SPDHG, and plot the predicted rate and the empirical rate in~\Cref{fig: last_fig}.
\begin{figure}[h]\label{fig: last_fig}
\begin{center}
\includegraphics[scale=0.4]{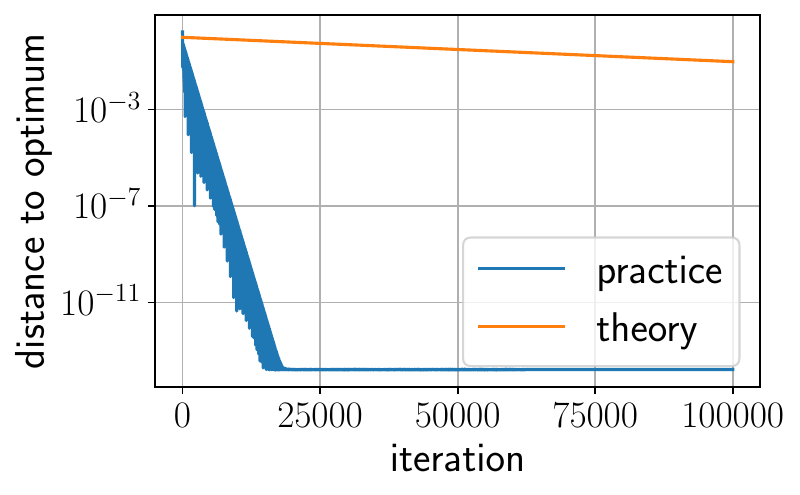}
\includegraphics[scale=0.4]{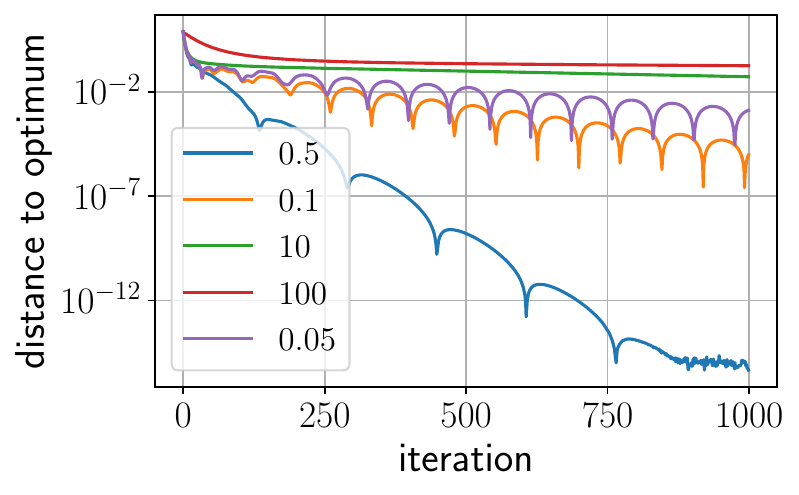}
\end{center}
\caption{left: empirical and theoretical linear rates, right: empirical rates with different $\mu$.}
\vskip -.2cm
\end{figure}
The resulting empirical rate is significantly faster than the worst case rate predicted by theory. 
We point out several possible explanations for this:
\begin{itemize}
\item Metric subregularity is too general to capture structures observed in practice.
\item Our step size choice is independent of metric subregularity constant, preventing optimizing the theoretical rate with respect to these quantities.
\end{itemize}

In fact, this phenomenon is not specific to our analysis and seems to be a common drawback of the existing analyses utilizing metric subregularity~\cite{latafat2019new}.
On this front, we observe that in our example, as $\mu$ increases, metric subregularity constant $\eta$ degrades.
However, as we see in the plot, the practical performance degrades when $\mu$ is either too big or too small (see~\Cref{fig: last_fig}).
This observation suggests that there might exist better regularity measures beyond metric subregularity that would help us derive better rates.
We believe that this is a promising future direction.

\section*{Acknowledgments} {\small Part of the work was done while A. Alacaoglu was at EPFL. This project has received funding from the European Research Council (ERC) under the European Union's Horizon 2020 research and innovation programme (grant agreement n° 725594 - time-data). This work was supported by the Swiss National Science Foundation (SNSF) under  grant number 407540\_167319. This project received funding from NSF Award  2023239;  DOE  ASCR  under  Subcontract  8F-30039  from  Argonne  National  Laboratory.

We are grateful to Panayotis Mertikopoulos, Ya-Ping Hsieh, and Yura Malitsky for discussions.
}

\section{Appendix}

\subsection{Proof of~\Cref{thm: conv_as}}\label{sec: full_proof_as_thm}

\begin{proof}
On~\eqref{eq: one_it_lem}, we pick $(x,y)=({x}^\star,y^\star)$ and by convexity, $D_g(x^k; z^\star) \geq 0$, $D_{f^\ast}(\hat{y}^{k+1}; z^\star) \geq 0$. Next, by using ${\Delta}^k = V_{k+1}(x^k - x^\star, y^{k+1}-y^\star)$, we
write~\eqref{eq: one_it_lem}
\begin{equation}\label{eq: before_rs_lemma}
\mathbb{E}_k \left[ {\Delta}^k \right] \leq \Delta^{k-1} - V(z^k - z^{k-1}).
\end{equation}
We denote $\bs q^k = (1\otimes x^k, 1\otimes y^k)$. By taking total expectation, summing~\eqref{eq: before_rs_lemma}, and using~\Cref{lem: defT}, we have $\sum_{k=1}^{\infty} \mathbb{E}\left[\| T(\bs q^{k-1}) - \bs q^{k-1} \|^2_{\bar{S}\bar{P}}\right] < +\infty$. We use Fubini-Tonelli theorem to exchange the infinite sum and the expectation 
 to obtain $\mathbb{E}\left[\sum_{k=0}^{\infty}\| T(\bs q^{k-1}) - \bs q^{k-1} \|^2_{\bar{S}\bar{P}}\right] < \infty$. Here, since $\sum_{k=0}^{\infty}\| T(\bs q^{k-1}) - \bs q^{k-1} \|^2_{\bar{S}\bar{P}}$ is nonnegative, we conclude that $\sum_{k=0}^{\infty}\| T(\bs q^{k-1}) - \bs q^{k-1} \|^2_{\bar{S}\bar{P}}$ is finite almost everywhere, which implies that $\| T(\bs q^{k-1}) - \bs q^{k-1} \|^2_{\bar{S}\bar{P}}$ converges to $0$ almost surely. Thus we established: $\exists \Omega_T$ with $\mathbb{P}(\Omega_T) = 1$ such that $\forall \omega \in \Omega_{T}$, we have $T(\bs q^k(\omega)) - \bs q^k(\omega) \to 0$.

We apply Robbins-Siegmund lemma~\cite[Theorem 1]{robbins1971convergence} on~\eqref{eq: before_rs_lemma} to get that a.s., $\Delta^k$ converges to a finite valued random variable and $V(z^k-z^{k-1})\to 0$. 
Consequently, by~\eqref{eq: v_lw_bd}, $\| y^k - y^{k-1}\|$ converges to $0$ a.s. 
Since a.s., $\Delta^k$ converges and $\| y^k - y^{k-1}\|$ converges to $0$, we have that $\| z^k - z^\star \|$ converges a.s.

In particular, we have shown that
\begin{equation}
\mathbb{P}\big( \omega \in \Omega\colon \lim_{k\to\infty} \|z_k(\omega)-z^\star\| \text{ exists.} \big) = 1.
\end{equation}
The probability $1$ set from which we select the trajectories is defined via $z^\star$. 
Let us denote the set
\begin{equation}
\Omega_{z^\star} = \Big\{ \omega \in \Omega\colon \lim_{k\to\infty} \|z_k(\omega)-z^\star \| \text{ exists.} \Big\}
\end{equation}
Thus our statement is actually: for each $z^\star\in\mathcal{Z}^\star$, there exists a set $\Omega_{z^\star}$ with probability 1, such that $\forall \omega\in\Omega_{z^\star}$, $\lim_{k\to\infty} \|z_k(\omega) - z^\star\|$ exists.

We now follow the arguments in~\cite[Proposition 2.3]{combettes2015stochastic},~\cite[Proposition 9]{bertsekas2011incremental},~\cite[Theorem 2]{iutzeler2013asynchronous},~\cite[Theorem 1]{fercoq2019coordinate} to strengthen this result.

Let us pick a set $\mathcal{C}$ which is a countable subset of $\mathrm{ri}(\mathcal{Z}^\star)$ that is dense in $\mathcal{Z}^\star$.
Let us denote the elements of $\mathcal{C}$ as $v_i$ for $i\in\mathbb{N}$.

We just proved that for all $v_i\in\mathcal{Z}^\star$, $\exists \Omega_{v_i}$ with $\mathbb{P}(\Omega_{v_i}) = 1$, such that $\forall \omega\in\Omega_{v_i}$, $\lim_{k\to\infty} \|z_k(\omega) - v_i\|$ exists.
Let us denote $\Omega_{\mathcal{C}} = \cap_{i\in\mathbb{N}} \Omega_{v_i}$.
As $\Omega_{\mathcal{C}}$ is the intersection of a countable number of sets of probability $1$, $\mathbb{P}(\Omega_{\mathcal{C}}) = 1$.

Next, we set $\tilde z\in\mathcal{Z}^\star$.
As $\mathcal{C}$ is dense in $\mathrm{ri}(\mathcal{Z}^\star)$, there exists a subsequence $v_{\phi(i)}$, where $\phi\colon \mathbb{N}\to\mathbb{N}$ is an increasing function, such that $v_{\phi(i)}\to \tilde{z}$.

We now pick $\omega\in \Omega_{\mathcal{C}}$ and study the existence of $\lim_{k\to\infty} \|z_k(\omega) - \tilde z\|$.
By triangle inequality, $\forall i\in\mathbb{N}$, 
\begin{align*}
\|z_k(\omega) - v_{\phi(i)}\| - \|v_{\phi(i)} - \tilde z\| &\leq \|z_k(\omega) - \tilde z\| \leq \|z_k(\omega) - v_{\phi(i)}\| + \|v_{\phi(i)} - \tilde z\|.
\end{align*}
Rearranging gives
\begin{equation*}
-\|v_{\phi(i)} - \tilde z\| \leq \|z_k(\omega) - \tilde z\|- \|z_k(\omega) - v_{\phi(i)}\| \leq \|v_{\phi(i)} - \tilde z\|.
\end{equation*}
As $\omega$ is chosen from $\Omega_{\mathcal{C}}$, and any element of $\Omega_{\mathcal{C}}$ is also an element of $\Omega_{v_i}$, we know that $\lim_{k\to\infty}\|z_k(\omega) - v_{\phi(i)}\|$ exists.
Moreover, recall that $v_{\phi(i)} \to \tilde z$.

We take limit as $k\to\infty$,
\begin{align*}
-\|v_{\phi(i)} - \tilde z\| &\leq \lim\inf_{k\to\infty} \|z_k(\omega) - \tilde z\| - \lim_{k\to\infty} \|z_k(\omega) - v_{\phi(i)}\| \\ 
&\leq \lim\sup_{k\to\infty} \| z_k(\omega) - \tilde z\| - \lim_{k\to\infty} \|z_k(\omega) - v_{\phi(i)}\|  \\
&\leq \|v_{\phi(i)} - \tilde z\|.
\end{align*}
As we take the limit along the subsequence defined by $\phi(i)$, we have $\lim_{i\to\infty} \|v_{\phi(i)} - \tilde z\| = 0$, which gives the equality of $\lim\inf$ and $\lim\sup$.

Thus, $\forall \omega \in \Omega_{\mathcal{C}}$ with $\mathbb{P}(\Omega_{\mathcal{C}})=1$ and $\forall \tilde z \in \mathcal{Z}^\star$, we have that  $\lim_{k\to\infty} \|z_k(\omega) - \tilde z\|$ exists.

We now pick $\omega\in\Omega_{\mathcal{C}}\cap \Omega_T$ and then as we have that $(z^k(\omega))_k$ is bounded, we denote by $\tilde{z}=(\tilde{x}, \tilde{y})$ one of its cluster points.
Then, we denote $\tilde{\bs q} = (1\otimes \tilde{x}, 1\otimes \tilde{y})$ and say that $\tilde{\bs q}$ is a cluster point of $(\bs q^k(\omega))_k$.

As $T(\bs q^k(\omega)) - \bs q^k(\omega) \to 0$, by continuity of $T$ we have $T(\tilde{\bs q})-\tilde{\bs q} \to 0$, therefore $\tilde{\bs q}$ is a fixed point of $T$.
We now use~\Cref{lem: defT} to argue that fixed points of $T$ which we denote as $(x_f(j), y_f(j))_{j=\{1, \dots, n\}}$ are such that $(x_f(j), y_f(j))\in\mathcal{Z}^\star, \forall j\in\{1, \dots, n\}$.
Since $\tilde{\bs q}$ is a fixed point of $T$, we conclude that $\tilde{z}\in\mathcal{Z}^\star$.

To sum up, we have shown that at least on some subsequence $z^k(\omega)$ converges to $\tilde{z}\in\mathcal{Z}^\star$.
Then, the result follows due to existence of the limit, proven earlier.
\end{proof}

\subsection{Proof of~\Cref{eq: lem_ergo_noexp}}\label{sec: lemma_for_erg}
\begin{proof}
As in~\cite{chambolle2018stochastic}, we use~\eqref{eq: haty_defin} to denote full dimensional updates.
By the definition of the proximal operator~\eqref{eq: prox_def} along with convexity of $f_i^\ast$ and $g$, we get, $\forall x \in\mathcal{X}$, $\forall y \in \mathcal{Y}$ and $\forall i = \{1, \dots, n\}$
\begin{align*}
g(x) &\geq g(x^{k}) + \langle x^{k}-x, A^\top \bar{y}^k \rangle + \frac{1}{2} \| x^{k}- x^{k-1} \|^2 _{\tau^{-1}} + \frac{1}{2} \| x^{k}-x\|^2_{\tau^{-1}}  \\  
&\qquad\qquad\qquad\qquad\qquad\qquad\qquad\qquad\qquad\qquad\qquad\qquad -\frac{1}{2} \| x - x^{k-1} \|^2_{\tau^{-1}},\\
f^\ast_i(y_i) &\geq f^\ast(\hat{y}^{k+1}_i) - \langle \hat{y}^{k+1}_i-y_i, A_i {x}^{k} \rangle + \frac{1}{2} \| \hat{y}^{k+1}_i- y^k_i \|^2 _{\sigma_i^{-1}} + \frac{1}{2} \| \hat{y}^{k+1}_i-y_i\|^2_{\sigma_i^{-1}} \\
&\qquad\qquad\qquad\qquad\qquad\qquad\qquad\qquad\qquad\qquad\qquad\qquad - \frac{1}{2} \| y_i - y^k_i \|^2_{\sigma_i^{-1}}.
\end{align*}
We sum the second inequality from $i=1$ to $n$ and add to the first inequality to obtain
\begin{align}
&0 \geq g(x^k) - g(x) + \langle x^k - x, A^\top \bar{y}^k \rangle + f^\ast(\hat{y}^{k+1}) - f^\ast(y) - \langle \hat{y}^{k+1}-y, Ax^k \rangle \notag \\
&+ \frac{1}{2} \left( -\|x^{k-1}-x\|^2_{\tau^{-1}} + \|x^k-x\|^2_{\tau^{-1}} + \| x^k - x^{k-1}\|^2_{\tau^{-1}} \right)  \notag \\
&+ \frac{1}{2} \big( -\|y^k-y\|^2_{\Dsigma^{-1}} + \|\hat{y}^{k+1}-y\|^2_{\Dsigma^{-1}} + \| \hat{y}^{k+1} - y^{k}\|^2_{\Dsigma^{-1}} \big). \label{eq: lem1_1}
\end{align}
We next note
\begin{align*}
&\mathcal{H}(x^k, \hat{y}^{k+1}; x, y) = g(x^k) + \langle Ax^k, y \rangle - f^\ast(y) - g(x) - \langle Ax, \hat{y}^{k+1} \rangle + f^\ast(\hat{y}^{k+1}),\notag \\
&\Delta_1 = \frac{1}{2} \big( - \| x^{k-1} - x \|^2_{\tau^{-1}} + \| x^k - x \|^2_{\tau^{-1}} + \| x^k - x^{k-1} \|^2_{\tau^{-1}} \big),\notag \\
&\Delta_2 = \frac{1}{2} \big( -\|y^k - y \|^2_{\Dsigma^{-1}} + \| \hat{y}^{k+1} - y \|^2_{\Dsigma^{-1}} + \| \hat{y}^{k+1} - y^k \|^2_{\Dsigma^{-1}}  \big).
\end{align*}
Then, we can write~\eqref{eq: lem1_1} as
\begin{align}\label{eq: new_erg_main_ineq1}
0 \geq \mathcal{H}(x^k, \hat{y}^{k+1}; x, y) + \langle A(x - x^k), \hat{y}^{k+1} - \bar{y}^k \rangle + \Delta_1 + \Delta_2.
\end{align}
We estimate by simple manipulations
\begin{align}
\mathcal{H}(x^k, &\hat{y}^{k+1}; x, y) = \mathcal{H}(x^k, {y}^{k+1}; x, y) + \langle Ax, y^{k+1}-\hat{y}^{k+1} \rangle + f^\ast(\hat{y}^{k+1}) - f^\ast(y^{k+1}) \notag\\
&- \left( f^\ast_{P^{-1} - I}(y^{k+1}) - f^\ast_{P^{-1}-I}(y^k) \right) + \left( f^\ast_{P^{-1} - I}(y^{k+1}) - f^\ast_{P^{-1}-I}(y^k) \right) \notag\\
&+\langle Ax, (P^{-1} - I)(y^{k+1} - y^k) \rangle - \langle Ax, (P^{-1}-I)(y^{k+1} - y^k) \rangle \notag\\
&=\mathcal{H}(x^k, {y}^{k+1}; x, y) + f^\ast(\hat{y}^{k+1}) - f^\ast(y^k) - (f^\ast_{P^{-1}}(y^{k+1})-f^\ast_{P^{-1}}(y^k)) \notag\\
&+\langle Ax, {y}^k - \hat{y}^{k+1} - P^{-1}(y^{k}-y^{k+1}) \rangle \notag \\
&+\left( f^\ast_{P^{-1} - I}(y^{k+1})-f^\ast_{P^{-1} - I}(y^k) \right) - \langle Ax, (P^{-1}-I)(y^{k+1} - y^k) \rangle \notag \\
&=\mathcal{H}(x^k, {y}^{k+1}; x, y) + f^\ast(\hat{y}^{k+1}) - f^\ast(y^k) - (f^\ast_{P^{-1}}(y^{k+1})-f^\ast_{P^{-1}}(y^k)) \notag\\
&+\langle Ax, {y}^k - \hat{y}^{k+1} - P^{-1}(y^{k}-y^{k+1}) \rangle +D_{f^\ast}^{P^{-1} - I}( y_{k+1}; y) - D_{f^\ast}^{P^{-1} - I}(y_{k}; y). \label{eq: gap_est}
\end{align}
By the definition of $\bar y^k$ in SPDHG, we have for the bilinear term in~\eqref{eq: new_erg_main_ineq1} that
\begin{align}
\langle A(x - x^k), \hat{y}^{k+1} &- \bar{y}^{k} \rangle = \langle A(x - x^k), \hat{y}^{k+1} - y^k - P^{-1}(y^k-y^{k-1}) \rangle \notag\\
&=\langle A(x - x^k), \hat{y}^{k+1} -  {y}^{k} \rangle - \langle A(x - x^{k-1}), P^{-1}( y^k - {y}^{k-1}) \rangle \notag\\ 
&- \langle A(x^{k-1} - x^{k}), P^{-1}(y^k - {y}^{k-1}) \rangle \notag\\
&=\langle A(x - x^k), P^{-1}( y^{k+1} - {y}^{k}) \rangle - \langle A(x - x^{k-1}), P^{-1}(y^k - {y}^{k-1}) \rangle \notag\\ 
&- \langle A(x^{k-1} - x^{k}), P^{-1}(y^k - {y}^{k-1}) \rangle \notag \\
&+ \langle A(x-x^k), \hat{y}^{k+1} - {y}^{k} -P^{-1}(y^{k+1} - y^{k}) \rangle. \label{eq: bilinear_est}
\end{align}
On $\Delta_2$, we add and subtract $\|y^k - y\|^2_{\Dsigma^{-1}P^{-1}} - \| y^{k+1} - y \|^2_{\Dsigma^{-1}P^{-1}}$ to get
\begin{equation}
-\Delta_2 = 
 -\frac{1}{2} \| y^{k+1} - y \|^2_{\Dsigma^{-1}P^{-1}} + \frac{1}{2} \| y^k - y \|^2_{\Dsigma^{-1}P^{-1}} -\frac{1}{2}  \| \hat{y}^{k+1} - y^k \|^2_{\Dsigma^{-1}}  + \epsilon^k,\label{eq: new_erg_delta2}
\end{equation}
where
\begin{align}
\epsilon^k &= \frac{1}{2} \Big[ \| y^k - y \|^2_{\Dsigma^{-1}} - \| \hat{y}^{k+1} - y \|^2_{\Dsigma^{-1}} - \big( \| y^k - y \|^2_{\Dsigma^{-1}P^{-1}} - \| y^{k+1} - y\|^2_{\Dsigma^{-1}P^{-1}} \big) \Big] \notag\\
&= \frac{1}{2} \Big[ \| y^k \|^2_{\Dsigma^{-1}} - \| \hat{y}^{k+1} \|^2_{\Dsigma^{-1}} -\big( \| y^k \|^2_{\Dsigma^{-1}P^{-1}} - \| y^{k+1}\|^2_{\Dsigma^{-1}P^{-1}} \big) \notag\\
&- 2\langle y, y^k - \hat{y}^{k+1} - P^{-1}( y^k - y^{k+1} ) \rangle_{\Dsigma^{-1}} \Big].
\end{align}
We use~\cref{eq: gap_est,eq: bilinear_est,eq: new_erg_delta2} in~\eqref{eq: new_erg_main_ineq1}, add and subtract $\frac{1}{2} \| y^k - y^{k-1} \|^2_{\Dsigma^{-1}P^{-1}}$ and use the definition $v^{k+1} = y^k - \hat{y}^{k+1} - P^{-1}(y^{k} - y^{k+1})$ from~\Cref{lem: new_erg_decoupling} to obtain
\begin{align}
\mathcal{H}(x^k, &{y}^{k+1}; x, y) \leq -\frac{1}{2} \| x^k - x \|^2_{\tau^{-1}} + \frac{1}{2} \| x^{k-1} - x \|^2_{\tau^{-1}} \notag \\
&- \langle A(x - x^k),  P^{-1}(y^{k+1}  - y^k) \rangle + \langle A(x - x^{k-1}), P^{-1}(y^k - y^{k-1}) \rangle \notag\\
&-\frac{1}{2} \| x^k - x^{k-1} \|^2_{\tau^{-1}} - \frac{1}{2} \| y^k - y^{k-1} \|^2_{\Dsigma^{-1}P^{-1}}   \notag \\
&- \langle A(x^k -x^{k-1}), P^{-1}(y^k - y^{k-1}) \rangle -\frac{1}{2} \| y^{k+1} - y \|^2_{\Dsigma^{-1}P^{-1}} \notag \\
&+ \frac{1}{2} \| y^k - y\|^2_{\Dsigma^{-1}P^{-1}} - \frac{1}{2}  \| \hat{y}^{k+1} - y^k \|^2_{\Dsigma^{-1}} + \frac{1}{2} \| y^{k}-y^{k-1}\|^2_{\Dsigma^{-1}P^{-1}}  \notag \\
&+\frac{1}{2} \left[ \| y^k \|^2_{\Dsigma^{-1}} - \| \hat{y}^{k+1} \|^2_{\Dsigma^{-1}} - \left(\|y^k\|^2_{\Dsigma^{-1}P^{-1}} - \| y^{k+1} \|^2_{\Dsigma^{-1}P^{-1}} \right) \right] \notag\\
&+f^\ast(y^k) - f^\ast(\hat{y}^{k+1}) -(f^\ast_{P^{-1}}(y^k) - f^\ast_{P^{-1}}(y^{k+1})) - \langle y, v^{k+1} \rangle_{\Dsigma^{-1}}\notag\\
&-\langle Ax^k, y^k - \hat{y}^{k+1} - P^{-1}(y^k - y^{k+1}) \rangle + D_{f^\ast}^{P^{-1}-I}(y_k; z) - D_{f^\ast}^{P^{-1}-I}(y_{k+1}; z).\label{eq: new_erg_main_ineq3}
\end{align}
The first result follows by the definitions of $V_k$ and $V$ from~\eqref{eq: v_vk_defin_main}, and definition of $\mathcal{E}^k$ from~\eqref{eq: e_def}.

On $\mathcal{E}^k$, we use $\mathbb{E}_k \left[ P^{-1}(y^{k} - y^{k+1}) \right] = y^k - \hat{y}^{k+1}$,
$\mathbb{E}_k\big[ f^\ast_{P^{-1}}(y^k) - f^\ast_{P^{-1}}(y^{k+1}) \big] = f^\ast(y^k) - f^\ast(\hat{y}^{k+1})$ and $\mathbb{E}_k \big[ \| y^{k+1} - y^k \|^2_{\Dsigma^{-1}P^{-1}} \big] = \| \hat y^{k+1} - y_k \|^2_{\Dsigma^{-1}}$
\begin{align*}
\mathbb{E}_k\left[ \mathcal{E}^k\right] &= -\frac{1}{2} \| \hat y^{k+1} - y^{k}\|^2_{\Dsigma^{-1}} + \frac{1}{2} \mathbb{E}_k \left[ \| {y}^{k+1}-y^k\|^2_{\Dsigma^{-1}P^{-1}}\right] \notag\\
&+\frac{1}{2} \left( \| y^k \|^2_{\Dsigma^{-1}} - \| \hat{y}^{k+1} \|^2_{\Dsigma^{-1}}\right) - \frac{1}{2}\mathbb{E}_k\left[\|y^k\|^2_{\Dsigma^{-1}P^{-1}} - \| y^{k+1} \|^2_{\Dsigma^{-1}P^{-1}} \right]  \notag\\
&+ f^\ast(y^k) - f^\ast(\hat {y}^{k+1})  - \mathbb{E}_k\left[f^\ast_{P^{-1}}(y^k) - f^\ast_{P^{-1}}(y^{k+1})\right]\notag\\
&-\langle Ax^k, y^k - \hat y^{k+1} - \mathbb{E}_k \left[ P^{-1}\left(y^k - y^{k+1}\right) \right] \rangle
=0.\notag
\end{align*}
\end{proof}

\subsection{Proof of~\Cref{lem: one_iteration}}\label{sec: pf_first_lem}
\begin{proof}
At step $k$ of SPDHG in Algorithm~\ref{alg: main_spdhg}, we select an index $i_k\in\{ 1, \dots, n \}$ randomly with probability $p_{i_k}$ and perform the following step on the dual variable
\begin{equation}\label{eq: lem1_stoc_dual_update}
y^{k+1}_{i_k} = \hat{y}^{k+1}_{i_k}, \text{ and } y^{k+1}_{i}=y^k_i, \forall i \neq i_k.
\end{equation}
For any $Y\in\mathcal{Y}$ that is measurable with respect to $\mathcal{F}_k$,~\eqref{eq: lem1_stoc_dual_update} immediately gives
\begin{align}
&\mathbb{E}_k [ y^{k+1} ] = P\hat{y}^{k+1} + \left( I-P\right) y^k, \label{eq: lem1_y_up1}\\
&\mathbb{E}_k \left[ \| y^{k+1}- Y \|^2_{\Dsigma^{-1}} \right] = \| \hat{y}^{k+1} - Y \|^2_{\Dsigma^{-1} P } +  \| y^k - Y \|^2_{\Dsigma^{-1} (I-P) }\label{eq: lem1_y_up2}.
\end{align}
A simple manipulation of~\eqref{eq: lem1_y_up1} and plugging in $Y = y$ and $Y=y_k$ in~\eqref{eq: lem1_y_up2} gives
\begin{align}
&\hat{y}^{k+1} = P^{-1} \mathbb{E}_k [y^{k+1}] - (P^{-1}-I) y^k \label{eq: y_exp_val}\\
&\| \hat{y}^{k+1} - y \|^2_{\Dsigma^{-1}} = \mathbb{E}_k \left[ \| y^{k+1}-y \|^2_{\Dsigma^{-1}P^{-1}}\right] - \| y^k - y\|^2_{\Dsigma^{-1}(P^{-1}-I)} \label{eq:expe_norms_y1} \\
&\| \hat{y}^{k+1} - y^k \|^2_{\Dsigma^{-1}} = \mathbb{E}_k \left[ \| y^{k+1}-y^k \|^2_{\Dsigma^{-1}P^{-1}} \right]. \label{eq:expe_norms_y2}
\end{align}
The first result follows by taking expectation of the result of~\Cref{eq: lem_ergo_noexp}, after using tower property and the above estimations.
On deriving the conclusion, we also use $D_g(x_k; z) + D_{f^\ast}(\hat y_{k+1}; z) = \mathcal{H}(x_k, \hat y_{k+1}; x, y)$ and \eqref{eq: gap_est}.

It is straightforward to prove~\eqref{eq: v_lw_bd} and~\eqref{eq: vk_lw_bd}. Since $y^k_{j} = y^{k-1}_j, \forall j\neq i_{k-1}$,
\begin{align}
\vert \langle Ax, &P^{-1}(y^k - y^{k-1}) \rangle \vert = \vert \langle A_{i_{k-1}}x, p_{i_{k-1}}^{-1}(y^k_{i_{k-1}} - y^{k-1}_{i_{k-1}}) \rangle \vert \notag \\
&\leq \|A_{i_{k-1}} x \| p_{i_{k-1}}^{-1} \| y^k_{i_{k-1}} - y^{k-1}_{i_{k-1}}\| \notag \\
&= \left(\tau^{1/2} \sigma_{i_{k-1}}^{1/2} p_{i_{k-1}}^{-1/2} \| A_{i_{k-1}} \|\right) \tau^{-1/2} \|x\| p_{i_{k-1}}^{-1/2}\sigma_{i_{k-1}}^{-1/2} \| y^k_{i_{k-1}} - y^{k-1}_{i_{k-1}}\| \notag \\
&\leq \gamma\left( \tau^{-1/2} \|x\| p_{i_{k-1}}^{-1/2}\sigma_{i_{k-1}}^{-1/2} \| y^k_{i_{k-1}} - y^{k-1}_{i_{k-1}}\| \right) \notag \\
&\leq \frac{\gamma}{2}\left( \|x\|^2_{\tau^{-1}} +  \| y^k_{i_{k-1}} - y^{k-1}_{i_{k-1}}\|^2_{p_{i_{k-1}}^{-1}\sigma_{i_{k-1}}^{-1}} \right) \notag \\
&= \frac{\gamma}{2}\left( \|x\|^2_{\tau^{-1}} +  \| y^k - y^{k-1}\|^2_{\Dsigma^{-1}P^{-1}} \right),\label{eq: pf_inner_product}
\end{align}
where the last step is due to $y^k_{j} = y^{k-1}_j, \forall j\neq i_{k-1}$.
Plugging in~\eqref{eq: pf_inner_product} into the definitions of $V(z^k - z^{k-1})$ and $V_k(z)$ is sufficient to prove~\eqref{eq: v_lw_bd} and~\eqref{eq: vk_lw_bd}.
\end{proof}

\begin{landscape}
\centering
\begin{table*}[t]
  \centering
  \begin{tabular}{| l | l | c | l | l |}
 \hline
      & Linear convergence & \makecell{Rates with \\only convexity} & \makecell{Step sizes for
      \\  linear convergence*} \\ \hline
    \cite{chambolle2018stochastic} & \makecell{$f_i^\ast:$ $\mu_i$-s.c. \\ $g:$ $\mu_g$-s.c.} & \makecell{{Ergodic $\mathcal{O}\left( \frac{1}{k} \right)$for}\\ {Bregman distance to solution}}  & $\|A_i\|, \mu_i, \mu_g$ \\ \hline
    \cite{zhang2017stochastic} & \makecell{$f_i^\ast:$ $\mu_i$-s.c. \\ $g:$ $\mu_g$-s.c.} & \makecell{Nonergodic $\mathcal{O}\left( \frac{1}{k} \right)$with\\ bounded dual domain and fixed horizon} & $\|A_i\|, \mu_i, \mu_g$ \\ \hline
        \cite{fercoq2019coordinate} & \makecell{$f_i^\ast:$ $\mu_i$-s.c. \\ $g:$ $\mu_g$-s.c.} & Randomly selected iterate $\mathcal{O}\left( \frac{1}{\sqrt{k}} \right)$  & {$n^2\tau\sigma_i\|A_i\|^2 < 1$} \\ \hline
    \cite{latafat2019new} & \cellcolor{lightgray}{$F$ is MS (see~\eqref{eq: kkt})}& $\times$ & {$n^2\tau\sigma_i\|A_i\|^2 < 1$} \\ \hline
    This paper & \cellcolor{lightgray}{{$F$ is MS (see~\eqref{eq: kkt})}} & \cellcolor{lightgray}{{Ergodic $\mathcal{O}\left( \frac{1}{k} \right)$ for primal-dual gap, objective values and feasibility }} & \cellcolor{lightgray}{$n\tau\sigma_i\|A_i\|^2 < 1$} \\
    \hline
  \end{tabular}
  \caption{\small{Comparison of primal dual coordinate descent methods. s.c. denotes strongly convex, MS denotes metrically subregular. Please see Section~\ref{sec: related_works} for a thorough comparison. Please see Section~\ref{sec: prelim} for comparison of MS and s.c. assumptions. $^\ast$Step sizes are for optimization with a potentially dense $A$ matrix and uniform sampling: $p_i = 1/n$.}}
  \label{tab:1}
\end{table*}

\begin{table*}[t]
  \centering
  \begin{tabular}{| c | c | c | c |}
 \hline
      & a.s. convergence & Linear convergence & Ergodic rates \\ \hline
    \cite{chambolle2018stochastic} & \makecell{$D_h(z^k; z^\star) \to 0$, for any $z^\star$ \\ where $D_h$ is Bregman \\ distance generated by \\ $h(z)=f^\ast(y)+ g(x)$} & \makecell{Assumption: $f_i^\ast, g$ s.c.\\ step sizes depending on $\mu_i, \mu_g$}  & \makecell{$D_h(z^k_{av}; z^\star)=\mathcal{O}(1/k)$ } \\ \hline
    This paper & $z^k \to z^\star$, for some $z^\star$. & \makecell{Assumption: $F$ in~\eqref{eq: kkt} is MS \\ Step sizes: $n\tau\sigma_i \|A_i\|^2 < 1^\ast$} & \makecell{ $\bullet$ Restricted primal-dual gap
    \\ $\mathbb{E}\left[G_{\mathcal{B}}(x^k_{av}, y^k_{av})\right] = \mathcal{O}(1/k)$ \\$\bullet$ $f$ is Lipschitz$^{\dagger}$ \\ $ \mathbb{E}\left[\vert P(x^k_{av}) - P(x^\star) \vert\right] = \mathcal{O}(1/k)$ \\ $\bullet$$f(\cdot)=\delta_b(\cdot)$ \\ $\mathbb{E}\left[\vert g(x^k_{av}) - g(x^\star) \vert\right] = \mathcal{O}(1/k)$ \\ $\mathbb{E}\left[\|Ax^k_{av}-b\|\right] =\mathcal{O}(1/k)$} \\
    \hline
  \end{tabular}
  \caption{\small{Comparison of our results and previous results on SPDHG. $^\ast$Step size condition is for uniform sampling: $p_i = 1/n$. $^\dagger$In this case $P(x) := f(Ax) + g(x)$. }}
  \label{tab:2}
\end{table*}
\end{landscape}


\bibliographystyle{plain}
\bibliography{literature_spdhg}

\end{document}